\documentclass[a4paper,12pt]{amsart}
\usepackage{multicol} 
\usepackage{newlfont}
\usepackage[all]{xy}
\usepackage{amsmath,amssymb,amsthm}

\numberwithin{equation}{section} 

\def\rit{\mathbb{R}}
\def\zit{\mathbb{Z}}   
\def\nit{\mathbb{N}}

\def\ppit{\mathbb{P}} 
\def\qit{\mathbb{Q}} 
\def\cit{\mathbb{C}} 
\def\fit{\mathbb{F}}

\newtheorem{theorem}{Theorem}[subsection]
\newtheorem{proposition}[theorem]{Proposition}
\newtheorem{lemma}[theorem]{Lemma}
\newtheorem{corollary}[theorem]{Corollary}

\theoremstyle{definition}
\newtheorem{definition}[theorem]{Definition}

\newtheorem{remark}[theorem]{Remark}
\newtheorem{example}[theorem]{Example}

\DeclareMathOperator{\gr}{gr}
\DeclareMathOperator{\tor}{torsion}
\DeclareMathOperator{\smal}{sm}
\DeclareMathOperator{\age}{age}
\DeclareMathOperator{\id}{id}
\DeclareMathOperator{\Diag}{Diag}
\DeclareMathOperator{\Pic}{Pic}
\DeclareMathOperator{\plat}{flat}
\DeclareMathOperator{\orb}{orb}
\DeclareMathOperator{\Hom}{Hom}

\begin{document}

\title[$QH^{\ast}(\ppit(w))$, a mirror differential system and their classical limits]{ \bf The small quantum cohomology of a weighted projective
  space, a mirror $D$-module and their classical limits}
\author{Antoine Douai *, Etienne Mann \dag}
\address{Antoine Douai, Laboratoire J.A Dieudonn\'e, UMR CNRS 6621, 
  Universit\'e de Nice, Parc Valrose, F-06108 Nice Cedex 2, France}
\email{Antoine.DOUAI@unice.fr}
\address{Etienne Mann,  D\'epartement de Math\'ematiques, Universit\'e
  de Montpellier 2, Place Eug\`ene Bataillon,  F-34 095 Montpellier
  CEDEX 5}
\email{etienne.mann@math.univ-montp2.fr}
\thanks{*Partially supported by ANR grant
    ANR-08-BLAN-0317-01/02 (SEDIGA)}
\thanks{\dag Partially
    supported by ANR grant TheorieGW}

\begin{abstract}
  We first describe a mirror partner ($B$-model) of the small quantum
  orbifold cohomology of weighted projective spaces ($A$-model) in the
  framework of differential equations: we attach to the $A$-model
  ({\em resp.} $B$-model) a quantum differential system 
(that is a trivial bundle equipped with a suitable flat meromorphic connection and a flat bilinear form) and we 
give an explicit isomorphism between these two quantum differential systems. On the $A$-side ({\em resp.} on the $B$-side), 
the quantum differential system alluded to is 
naturally produced by the small quantum cohomology ({\em resp.} a solution of the Birkhoff problem for the Brieskorn lattice of a Landau-Ginzburg model). 
Then we study the degenerations of these quantum differential systems 
and we apply our results to the construction of (classical, limit, logarithmic) Frobenius
  manifolds.
\end{abstract}

\maketitle

\tableofcontents

\numberwithin{theorem}{section}

\section{Introduction}

Mirror symmetry has different mathematical formulations: equality
between the $I$ and $J$ functions, equivalence of categories,
isomorphisms of Frobenius manifolds {\em etc}... In this paper, we
first explore the differential aspect of this symmetry for weighted projective spaces
$\ppit (w):=\ppit (w_{0},w_{1},\cdots ,w_{n})$, the $A$-model, where
$w_{0},w_{1},\cdots ,w_{n}$ are positive integers (to simplify the
exposition, we will assume that $w_{0}=1$). It will be encoded by the
{\em quantum differential system} on $\ppit^{1}\times M$, that is
tuples $(M,H,\nabla,S)$ where $M$ is a complex manifold, $H$ is a
trivial bundle on $\ppit^{1}\times M$, $\nabla$ is a flat meromorphic
connection with logarithmic poles at $\{\infty \}\times M$ and with poles of
order less or equal to two at $\{0\}\times M$, and $S$ is a symmetric, nondegenerate,
$\nabla$-flat bilinear form (for short a {\em metric}, even if there
is no positivity consideration here).  More precisely, we attach a
quantum differential system on $\ppit^{1}\times \cit^{*}$ to the small quantum
orbifold cohomology of $\ppit (w)$ and we show that it is isomorphic
to the one associated with a suitable regular function (the {\em Landau-Ginzburg} model): this $B$-model will
be our mirror partner for the small quantum orbifold cohomology of
weighted projective spaces. 

The reason to work with quantum differential systems is very natural: first, on the $A$-side, they
arise classically as a ``completion'' of the quantum product to an absolute flat connection (thanks to Dubrovin's formalism), and we cannot expect much better.
 Second, on the $B$-side ({\em i.e} in singularity theory), construction of quantum differential systems, independently of mirror symmetry, 
is a long story (it goes back to K. Saito \cite{SK} and his theory of primitive forms) and has 
motivated a lot of work: general statements in our framework (global case) can be found in \cite{DoSa1} (where one of the main tool is Hodge theory) and
some significant class of examples or situations are studied for instance in \cite{dGMS}, \cite{Do1}, \cite{DoSa2}, \cite{RS}. 

It is then reasonable to compare such objects, appearing in quite different areas of mathematics: 
in particular, this enables us to understand the results of \cite{Coa} in the light of singularity theory. 
While computations of 
quantum differential systems are not so easy in general, they can be explicitely done in our situation.
This strategy could be useful in order 
to study more generally the case of the small quantum cohomology 
of hypersurfaces (or complete intersections) in (weighted) projective spaces, for which the Landau-Ginzburg models are clearly 
identified (see \cite{givental2}, \cite{Hori}) and not so far from the ones considered here.

In order to get this first result, we proceed as follows: following
Iritani \cite{Ir}, we first attach a quantum differential system to any proper
smooth Deligne-Mumford stack using the quantum orbifold cohomology.
Thanks to the results recently obtained in \cite{Coa}, this
construction can be done very explicitely in the case of weighted
projective spaces and yields, taking into account an action of the
Picard group, a quantum differential system
$${\cal Q}^{A}=(\mathcal{M}_{A}, \widetilde{H}^{A,\smal}, \widetilde{\nabla}^{A, \smal}, \widetilde{S}^{A, \smal}, n)$$
where
$\mathcal{M}_{A}=H^{2}(\ppit(w),\cit)/\Pic(\ppit(w))\simeq\cit^{*}$,
the metric $\widetilde{S}^{A, \smal}$ being constructed with the help
of the orbifold Poincar\'e duality. We will call this quantum differential system
the {\em (small) $A$-model quantum differential system}. It should be noticed, and
this will be a crucial observation, that the usual sections
$\mathbf{1}_{f_{i}}P^{j}$ of the orbifold cohomology are not {\em global} sections of the
bundle $\widetilde{H}^{A,\smal}$ whereas the $P^{\bullet j}$'s (iteration $j$-times of $P=c_{1}(O(1))$ under the small quantum product) are
global sections of it, see Remark \ref{rem:global,sections}.

We then look for a mirror partner of this $A$-model quantum differential system. 
Using the methods developed in \cite{DoSa2} and \cite{Man}, we show how it is obtained from the Gauss-Manin system of the function 
(this is our ``Landau-Ginzburg'' model)
$F:U\times\mathcal{M}_{B}\rightarrow \cit$ defined by
$$F(u_{1},\cdots ,u_{n},x)=u_{1}+\cdots +u_{n}+\frac{x}{u_{1}^{w_{1}}\cdots u_{n}^{w_{n}}}$$
where $U=(\cit^{*})^{n}$ and $\mathcal{M}_{B}=\cit^{*}$. Indeed, a solution of the Birkhoff problem for the Brieskorn lattice of $F$ gives a trivial bundle $H^{B}$ on $\ppit^{1}\times \mathcal{M}_{B}$ equipped with a connection with the desired poles.
 We get in this way (see section \ref{deformation}) a quantum differential system  
$${\cal Q}^{B}=(\mathcal{M}_{B}, H^{B}, \nabla^{B}, S^{B}, n)$$
using a distinguished solution of the Birkhoff problem, closely related with the canonical ones defined in \cite{DoSa2} in the case $x=1$. 
This will be our {\em $B$-model quantum differential system}.

We prove that the quantum differential systems ${\cal Q}^{A}$ and ${\cal Q}^{B}$ are isomorphic: the isomorphism is very explicit 
and identifies the sections $P^{\bullet j}$ ({\em resp.}
$\mathbf{1}_{f_{i}}P^{j}$) in terms of suitable sections
of the Brieskorn lattice of $F$ (Theorem \ref{quantum}). 
At the end, we get an answer to the following question, which was one of the first (chronologically) motivations of this work: what should the mirror partner of
the standard (orbifold) cohomology basis be? 
We discuss the comparison between our result and Proposition 4.8 of Iritani \cite{Ir}
in Remark \ref{rem:correspondence}.\\

 Identifying these two models, we obtain finally a quantum differential system 
\begin{displaymath}
  \mathcal{S}_{w}=(\mathcal{M},H,\nabla,S,n)
\end{displaymath}
where $\mathcal{M}=\cit^{*}$ (the index $w$ recalls the weights $w_{0},\cdots ,w_{n}$) and, as a by-product, a  Frobenius type structure $\fit_{w}$ on $\mathcal{M}$ in the sense of \cite{Do1} and \cite{HeMa}, that is a tuple 
$$\fit_{w}=(\mathcal{M}, E, R_{0}, R_{\infty}, \Phi, \bigtriangledown ,g)$$
the different objects involved satisfying some natural compatibility relations (coming from the flatness of $\nabla$). 
This Frobenius type structure will be the main tool in our construction of Frobenius manifolds.

In the second part of this paper, we study the behaviour of these structures at the origin (this kind of problem is also considered in \cite{dGMS}, using another strategy and in a different situation). 
We construct in section \ref{limitSTF} a limit quantum differential system (and thus a limit Frobenius type structure
$\overline{\fit}_{w}$)
$$\overline{\mathcal{S}}_{w}=(\overline{H},\overline{\nabla}, \overline{S},n)$$ 
on $\ppit^{1}$ using Deligne's canonical extensions of the connection involved. 
We explain how it can be used to understand the correspondence between ``classical limits'', that is between the orbifold cohomology ring of $\ppit (w)$ and a 
suitable graded vector space: we hope that it will shed new light on \cite[theorem 1.1]{Man}.

The last part is devoted to the construction of
classical, limit and logarithmic Frobenius manifolds: we need a
Frobenius type structure and a section of the corresponding bundle
such that the associated period map is invertible, in other words a
{\em primitive} section, see for instance \cite[Chapitre VII]{Sab1}.
To get such objects, we look, following \cite{Do1} and \cite{HeMa},
for unfoldings of the initial data $\fit_{w}$ (in the classical case) and $\overline{\fit}_{w}$ (in the limit case): they will be obtained from unfoldings
of the quantum differential systems $\mathcal{S}_{w}$ and
$\overline{\mathcal{S}}_{w}$ (another reason to work with quantum differential systems
is that one can unfold them, see \S \ref{construction}). In the best cases, we
use the reconstruction method presented in {\em loc. cit.} to get {\em
  universal} unfoldings.  We
show first and in this way that
\begin{enumerate}
\item the Frobenius type structure $\fit_{w}$ yields a Frobenius manifold on $\Delta\times (\cit^{\mu -1}, 0)$, $\Delta$ denoting any open disc in $M$. We will use it to compare, using the arguments given in \cite{Do1}, the canonical Frobenius manifolds attached to the  functions $F_{x}:=F(.\ ,x)$, $x\in\Delta$, by the punctual construction given in \cite{DoSa2}; 
\item the limit Frobenius type structure $\overline{\fit}_{w}$ yields
  ``limit'' Frobenius manifolds, depending on the weights
  $w_{0},\cdots ,w_{n}$ . For instance, we get a universal unfolding
  only in the manifold case ({\em i.e} $w_{0}=\cdots =w_{n}=1$) and, as a
  consequence of the universality, we obtain a unique, up to
  isomorphism, limit Frobenius manifold. In the orbifold case,
  that is if there is a weight $w_{i}$ greater or equal to two, we
  construct a limit Frobenius manifold for which the product is
  constant, but we loose any kind of unicity: our limit Frobenius type
  structure could produce other Frobenius manifolds, which can be
  difficult to compare.
\end{enumerate}

This distinction between the manifold case and the orbifold case also appears in the construction of {\em logarithmic} Frobenius manifolds. For instance, 
in the manifold case, we show how our initial data $\fit_{w}$ yields more precisely, as before {\em via} one of its universal unfoldings, a logarithmic Frobenius 
manifold with logarithmic pole along $x=0$ in the sense of \cite{R}. This gives the logarithmic Frobenius manifold attached to $\ppit^{n}$ in {\em loc. cit.} by a 
different method (Reichelt works directly with the whole Gromov-Witten potential;  more generally, he constructs a logarithmic Frobenius manifold from the big quantum 
cohomology of any smooth manifold). In the orbifold case,  our metric degenerates at the origin and we get only a logarithmic Frobenius manifold {\em without metric}. 
The construction of a logarithmic Frobenius manifold using this method is still an open problem. We also explain why Reichelt's construction does not work in the orbifold case.\\

The paper is organized as follows: we define the quantum differential systems and the Frobenius type structures in
section \ref{FrobeniusSaito}.  The construction of the quantum differential system
attached to an orbifold (the $A$-model quantum differential system) is done in section
\ref{Amodel}. It is explained in the case of the weighted projective
spaces.  Section \ref{Bmodel} is devoted to the construction of the
$B$-model quantum differential system and the main theorem is stated in section
\ref{mirrorsmall}. We compute the limits of our structures in section
\ref{limite} and we 
discuss the construction of Frobenius manifolds in section \ref{construction}.\\

This paper is a revised version of the preprint \cite{Do2} and  supersedes it.\\

\section{Quantum differential systems and Frobenius type structures}
\label{FrobeniusSaito}

\begin{definition}\label{defi:Saito,structure} Let $M$ be a complex manifold, $n$ be a positive integer. A {\em quantum differential system of weight $n$}\footnote{A quantum differential system is also sometimes called a $tr(TLEP)(n)$-structure, see  \cite[Section 5.2]{He}}
 on $\ppit^{1}\times M$ is a tuple $(M, H, \nabla, S, n)$ where 
\begin{itemize}
 \item $H$ is a trivial bundle over $\ppit^{1}\times M$,
 \item $\nabla$ is a meromorphic, flat connection on $H$ with poles
   along $\{0,\ \infty \}\times M$, logarithmic along
   $\{\infty\}\times M$, of order less or equal to $2$ along
   $\{0\}\times M$; this implies that the connection has locally the form 
   \begin{align}\label{eq:log+rk1,connection}
     \nabla&:=d +\left(\frac{A_{-1}^{(0)}(\underline{q})}{z}+A_{0}^{(0)}(\underline{q})\right)\frac{dz}{z}+
\sum_{i=1}^{s}\left(\frac{A_{-1}^{(i)}(\underline{q})}{z}+A_{0}^{(i)}(\underline{q})\right)dq_{i}
   \end{align}
   where $z$ is a coordinate on $\ppit^{1}$, $\underline{q}=(q_{1}, \ldots ,q_{s})$
   are coordinates on $M$ and the matrices involved are holomorphic in $q_{1},
   \ldots ,q_{s}$.
 \item $S$ is a $\nabla$-flat, nondegenerate $\cit$-bilinear form,
   satisfying
$$S:\mathcal{H}\times i^{\ast}\mathcal{H}\to z^{n}\mathcal{O}_{\ppit^{1}\times M}$$
where $\mathcal{H}$ is the sheaf of sections of $H$, $z$ is a fixed
coordinate on $\ppit^{1}\smallsetminus\{\infty\}$ and
$$i:\ppit^{1}\times M\to \ppit^{1}\times M$$
sends $(z,\underline{t})$ to $(-z,\underline{t})$.
\end{itemize}
\end{definition}

\begin{definition} Two quantum differential systems $(M_{1},
  H_{1}, \nabla_{1}, S_{1}, n_{1})$ and $(M_{2}, H_{2}, \nabla_{2},
  S_{2}, n_{2})$ are {\em isomorphic} if there exists an isomorphism
  $(\id,\tau):\ppit^{1}\times M_{1}\to \ppit^{1}\times M_{2}$ and an
  isomorphism of vector bundles $\gamma :H_{1}\to (\id,\tau)^{*}H_{2}$
  compatible with the connections and the metrics, {\em i.e} such that
  \begin{itemize}
  \item  $\nabla_{2}^{*}\gamma (s)=\gamma (\nabla_{1}s)$ for any section $s$ of $H_{1}$,
  \item $S_{2}^{*}(\gamma (e), \gamma (f))=S_{1}(e,f)$ for any sections $e$ and $f$ of $H_{1}$ (in particular $n_{1}=n_{2}$), 
  \end{itemize}
$\nabla_{2}^{*}$ ({\em resp.} $S_{2}^{*}$) denoting the connection ({\em resp.} the metric) on $(\id,\tau)^{*}H_{2}$
induced by $\nabla_{2}$ ({\em resp.} $S_{2}$).  
\end{definition}

\begin{definition} A {\em Frobenius type structure}\footnote{This terminology is borrowed from \cite{HeMa}} on $M$ is a tuple
$$(M, E, \bigtriangledown , R_{0}, R_{\infty},\Phi ,g)$$
\noindent where
\begin{itemize}
\item $E$ is a locally free sheaf of ${\cal O}_{M}$-modules,
\item $\bigtriangledown$ is a connection on $E$, 
\item  $R_{0}$ and $R_{\infty}$ are ${\cal O}_{M}$-linear endomorphisms of $E$,
\item $\Phi :E\rightarrow \Omega^{1}(M)\otimes E$ is a ${\cal
    O}_{M}$-linear map,
\item $g$ is a ${\cal O}_{M}$-bilinear form, symmetric and nondegenerate (a {\em metric}) on the sheaf of sections of $E$, 
\end{itemize}
these objects satisfying the relations

\begin{center} $\bigtriangledown^{2}=0$, $\bigtriangledown (R_{\infty})=0$, $\Phi\wedge\Phi =0$, $[R_{0},\Phi ]=0$,
\end{center}
\begin{center} $\bigtriangledown (\Phi )=0$, $\bigtriangledown (R_{0})+\Phi =[\Phi ,R_{\infty}]$, 
\end{center} 
\begin{center}
 $\bigtriangledown (g)=0$, $\Phi^{*}=\Phi$, $R_{0}^{*}=R_{0}$,
 $R_{\infty}+R_{\infty}^{*}=r \id$
\end{center}
for a suitable constant $r$, $^{*}$ denoting as above the adjoint with respect to $g$. 
\end{definition}

\begin{remark}\label{FTSonapoint} 
(1) A quantum differential system on $\ppit^{1}$ ({\em i.e} $M=\{point \}$) will be denoted by $(H, \nabla, S, n)$.\\ 
(2) A Frobenius type structure on a point is a tuple
$$(E,R_{0}, R_{\infty}, g)$$
where $E$ is a finite dimensional vector space over $\cit$,  $g$ is a symmetric and nondegenerate bilinear form on $E$, $R_{0}$ and $R_{\infty}$ 
being two endomorphisms of $E$ satisfying $R_{0}^{*}=R_{0}$ and
$R_{\infty}+R_{\infty}^{*}=r \id$ for a suitable complex number $r$, $^{*}$ denoting the adjoint with respect to $g$. $\blacklozenge$
\end{remark}

A quantum differential system yields a Frobenius type structure (see for instance \cite[VI, paragraphe 2.c p.214]{Sab1}). Indeed,  
let $(M, H, \nabla ,S, n)$ be a quantum differential system on $\ppit^{1}\times M$, $\sigma_{1},\cdots ,\sigma_{r}$ be a basis of global sections of $H$.
Define 
\begin{itemize}
 \item $E:=H\mid_{\{0\}\times M}$ and $E_{\infty }:=H\mid_{\{\infty\}\times M}$ ($E$ and $E_{\infty}$ are canonically isomorphic),
\item $R_{0}[\sigma_{i}]:=[z^{2}\nabla_{\partial_{z}}\sigma_{i}],$ for $i=1,\cdots, r$ where $[\cdot ]$ denotes  the class in $E$, 
\item $g([\sigma_{i}],[\sigma_{j}]):=z^{-n}S(\sigma_{i},\sigma_{j})$ for $i,j=1,\cdots,r$, 
\item $\Phi_{\xi}[\sigma_{i}]:=[z\nabla_{\xi}\sigma_{i}]$
for any vector field $\xi$ on $M$. 
\end{itemize}
The connection $\bigtriangledown$ and the endomorphism $R_{\infty}$ are defined analogously, using the restriction $E_{\infty}$: we put, with $\tau =z^{-1}$,
\begin{itemize}
\item $R_{\infty}[\sigma_{i}]:=[\nabla_{\tau\partial_{\tau}}\sigma_{i}]$
\item $\bigtriangledown_{\xi}[\sigma_{i}]:=[\nabla_{\xi}\sigma_{i}].$
\end{itemize}

\begin{proposition}[see \cite{Sab1}]
\label{MStoSTF}
The tuple $(M, \bigtriangledown, E, R_{0}, R_{\infty}, \Phi ,  g)$ is a Frobenius type structure on $M$.
\end{proposition}

\noindent Notice that the characteristic relations of a Frobenius type
structure is the counterpart of the integrability of the connection of
the associated quantum differential system.

\numberwithin{theorem}{subsection}

 \section{$A$-model}\label{Amodel}

 Let $\mathcal{X}$ be a smooth proper Deligne-Mumford stack of finite
 type over $\cit$ of complex dimension $n$.  In this section, we
 construct a quantum differential system on $\ppit^{1}\times M_{A}$ where
 $M_{A}:=H^{*}_{\orb}(\mathcal{X},\cit)$ (a quantum $D$-module in the
 sense of \cite{Ir}; a similar notion, called \textit{semi-infinite
   variation of Hodge structure} is defined by Barannikov in \cite{Ba}
 and \cite{Ba2}). This will be our big $A$-model quantum differential system.  We
 restrict it to $H^{2}(\mathcal{X},\cit)$ and we quotient the result
 by an action of the Picard group of $\mathcal{X}$ to get the small
 $A$-model quantum differential system.  Finally, we explain this construction for
 weighted projective spaces.

Our general references will be \cite{LMB} and \cite[Appendix]{Vistoli}
for Deligne-Mumford stacks and \cite{AGV}, \cite{CR1} and \cite{CR2}
for orbifold cohomology.

\subsection{The big $A$-model quantum differential system}

First, we recall some basic facts about orbifold cohomology.  The
\textit{inertia stack}, denoted by
$\mathcal{IX}:=\mathcal{X}\times_{\mathcal{X}\times \mathcal{X}}
\mathcal{X}$, is the fiber product over the two diagonal morphisms
$\mathcal{X}\to \mathcal{X}\times \mathcal{X}$. The inertia stack is a
smooth Deligne-Mumford stack but different components will in general
have different dimensions.  The identity section gives an irreducible
component which is canonically isomorphic to $\mathcal{X}$. This
component is called \textit{the untwisted sector}. All the other
components are called \textit{twisted sectors}. We thus have
$$\mathcal{IX}=\mathcal{X}\sqcup \bigsqcup_{v\in T} \mathcal{X}_{v}$$
where $T$ parametrizes the set of components of the twisted sectors of
$\mathcal{IX}$.

The orbifold cohomology of $\mathcal{X}$ is defined, as vector space,
by $H^{*}_{\orb}(\mathcal{X},\cit):=H^{*}(\mathcal{IX},\cit)$.  
We have
$$H^{*}_{\orb}(\mathcal{X},\cit)=H^{*}(\mathcal{X},\cit)\oplus \bigoplus_{v \in T}H^{*}(\mathcal{X}_{v},\cit).$$
We will put $M_{A}:=H^{*}_{\orb}(\mathcal{X},\cit)$ in what follows. 

To define a grading on $M_{A}$, we associate to any $v\in T$ a
rational number called the \textit{age} of $\mathcal{X}_{v}$.  A
geometric point $(x,g)$ in $\mathcal{IX}$ is a point $x$ of
$\mathcal{X}$ and $g\in Aut(x)$. Fix a point $(x,g)\in
\mathcal{X}_{v}$. As $g$ acts on the tangent space $T_{x}\mathcal{X}$,
we have an eigenvalue decomposition of $T_{x}\mathcal{X}$.  For any $f
\in [0,1[$, we denote $(T_{x}\mathcal{X})_{f}$ the sub-vector space
where $g$ acts by multiplication by $\exp(2\sqrt{-1}\pi f)$.  We
define
\begin{displaymath}
  \age(v):=\sum_{f\in[0,1[ } f.\dim_{\cit}(T_{x}\mathcal{X})_{f}.
\end{displaymath}
This rational number only depends on $v$.  Let $\alpha_{v}$ be a
homogeneous cohomology class of $\mathcal{X}_{v}$.  We define the {\em orbifold degree} of $\alpha_{v}$ by
\begin{displaymath}
  \deg^{\orb}(\alpha_{v}):=\deg(\alpha_{v})+2\age(v).
\end{displaymath}

Let $\phi_{0},\ldots ,\phi_{N}$ be a graded homogeneous basis of
$H^{*}_{\orb}(\mathcal{X},\qit)$ such that $\phi_{0}\in
H^{0}(\mathcal{X},\qit)$ and $\phi_{1}, \ldots ,\phi_{s}\in
H^{2}(\mathcal{X},\qit)$.  Notice that the cohomology classes
$\phi_{1}, \ldots ,\phi_{s}$ are in the cohomology of $\mathcal{X}$
{\em i.e} in the cohomology of the untwisted sector.  
We will denote by $\underline{t}:=(t_{0}, \ldots, t_{N})$ the coordinates of $M_{A}$ associated to this basis.

\subsubsection{The trivial bundle and the flat meromorphic connection}

Let $H^{A}$ be the trivial vector bundle over $\ppit^{1}\times M_{A}$ whose
fibers are $H^{*}_{\orb}(\mathcal{X},\cit)$.  For $i\in\{0, \ldots
,N\}$, we see $\phi_{i}$ as a global section of the bundle $H^{A}$.

Define the vector field, called \textit {the Euler vector field},
\begin{displaymath}
  \mathfrak{E}:=\sum_{i=0}^{N}\left(1-\frac{\deg^{\orb}(\phi_{i})}{2}\right)t_{i}\partial_{i} +
  \sum_{i=1}^{s}r_{i}\partial_{i}.
\end{displaymath}
where the $r_{i}$ are rational numbers determined by the equality
$c_{1}(T\mathcal{X})=\sum_{i=1}^{s}r_{i}\phi_{i}$ and $\partial_{i}$ denotes the vector field $\frac{\partial}{\partial t_{i}}$.

The big quantum product\footnote{Usually, working on quantum
  cohomology, one has either to add the Novikov ring (see section
  8.1.3 of \cite{CK}) or to assume that the quantum product converges
  on some open of $M_{A}$ (see Assumption 2.1 in \cite{Ir}). But we
  will mainly consider the small quantum product of weighted
  projective spaces, for which the convergence problems are solved.} denoted by $\bullet_{\underline{t}}$,
endows the sheaf of sections of the vector bundle $H^{A}$ with a product.  We define a
$\mathcal{O}_{M_{A}}$-linear homomorphism which will turn out to be a
Higgs field ({\em ie.} $\Phi\wedge \Phi=0$ see Proposition
\ref{prop:flat,connection})
$$\Phi: TM_{A} \rightarrow End\left(H^{A}\right) \mbox{ by }\Phi (\partial_{i})=\phi_{i}\bullet_{\underline{t}}.$$
In coordinates, we have
\begin{displaymath}
  \Phi=\sum_{i=0}^{N} \Phi^{(i)}(\underline{t}) dt_{i} 
\end{displaymath}
where $\Phi^{(i)}(\underline{t})$ is the endomorphism
$\phi_{i}\bullet_{\underline{t}}$.

Define, on the trivial bundle $H^{A}$, the connection
\begin{equation*}
  \nabla^{A} :=d_{M_{A}}+d_{\ppit^{1}} - \frac{1}{z}\pi^{\ast}\Phi + \left(\frac{1}{z}\Phi(\mathfrak{E})+R_{\infty}\right)\frac{dz}{z}
\end{equation*}
where $\pi:\ppit^{1}\times M_{A} \to M_{A}$ is the projection and $R_{\infty}$ is the
semi-simple endomorphism whose matrix in the basis $(\phi_{i})$ is
$$R_{\infty}=\Diag\left(\frac{\deg^{\orb}(\phi_{0})}{2}, \ldots ,\frac{\deg^{\orb}(\phi_{N})}{2}\right).$$

\noindent The proposition below is well-known. Some parts
and ideas of the proof can be found in \cite{Sab1},\cite{He},\cite{Manin} and \cite{CK}.

\begin{proposition}[see \S 2.2 in \cite{Ir}]
  \label{prop:flat,connection} The meromorphic connection $\nabla^{A}$
  is flat.
\end{proposition}

\subsubsection{The pairing} 

The vector space $H^{*}_{\orb}(\mathcal{X},\cit)$ is endowed with a
nondegenerate pairing which is called the orbifold Poincar\'e pairing
(see \cite{CR2}).
We denote it by $\langle \cdot, \cdot \rangle$.  It satisfies the
following homogeneity property:
\begin{equation}\label{eq:orbifold,PD}
  \mbox{ if }\langle \phi_{i},\phi_{j}\rangle \neq 0\ \mbox{ then } \deg^{\orb}(\phi_{i})+\deg^{\orb}(\phi_{j})=2n
\end{equation}
where $n=\dim_{\cit}\mathcal{X}$. We define a pairing $S^{A}$ on the global sections $\phi_{0}, \ldots
,\phi_{N}$ of $H^{A}$ by 
\begin{displaymath}
  S^{A}(\phi_{i},\phi_{j}):= z^{n}\langle \phi_{i},\phi_{j}\rangle.
\end{displaymath}
and we extend it by linearity using the rules 
\begin{equation}\label{eq:15}
  a(z,\underline{t})S^{A}(\cdot,\cdot)=S^{A}(a(z,\underline{t})\cdot,\cdot)=S^{A}(\cdot,a(-z,\underline{t})\cdot)
\end{equation}
for any
$a(z,\underline{t})\in \mathcal{O}_{\ppit^{1}\times M_{A}}$.

\begin{proposition}\label{prop:A,metric}
  The pairing $S^{A}(\cdot,\cdot)$ is nondegenerate, $(-1)^{n}$-symmetric and $\nabla^{A}$-flat.
\end{proposition}

\begin{proof} As the orbifold Poincar\'e duality is nondegenerate, the pairing
 $S^{A}$ is nondegenerate and $(-1)^{n}$-symmetric by \eqref{eq:15}.
The $\nabla^{A}$-flatness is equivalent to
\begin{align}
 z \partial_{z}S^{A}(\phi_{i},\phi_{j})=S^{A}(\nabla^{A}_{z\partial_{z}}\phi_{i},\phi_{j})+S^{A}(\phi_{i},\nabla^{A}_{z\partial_{z}}\phi_{j})\label{eq:13}
 \\
\partial_{k}S^{A}(\phi_{i},\phi_{j})=S^{A}(\nabla^{A}_{\partial_{k}}\phi_{i},\phi_{j})+S^{A}(\phi_{i},\nabla^{A}_{\partial_{k}}\phi_{j})\label{eq:14}
\end{align}
\noindent Using the rules \eqref{eq:15}, we have 
\begin{align*}
  z\partial_{z}S^{A}(\phi_{i},\phi_{j})&=nS^{A}(\phi_{i},\phi_{j}) \\
  S^{A}(z\nabla^{A}_{\partial_{z}}\phi_{i},\phi_{j}) &=
  \frac{1}{z}S^{A}(\Phi(\mathfrak{E})(\phi_{i}),\phi_{j})+S^{A}(R_{\infty}
  \phi_{i},\phi_{j})\\
S^{A}(\phi_{i},\nabla^{A}_{z\partial_{z}}\phi_{j})&=-\frac{1}{z}S^{A}(\phi_{i},\Phi(\mathfrak{E})(\phi_{j}))+S^{A}(\phi_{i},R_{\infty} \phi_{j})
\end{align*}

We denote by $R_{\infty}^{\ast}$ the adjoint of $R_{\infty}$ with respect
to $S^{A}(\cdot,\cdot)$. The following equalities (which follow from \cite[\S 7.6]{AGV})
\begin{align}\label{eq:16}
\langle \phi_{k}\bullet_{\underline{t}}\phi_{i},\phi_{j}\rangle& =\langle
\phi_{i},\phi_{k}\bullet_{\underline{t}}\phi_{j}\rangle\\
R_{\infty}+R_{\infty}^{*}&=n \id
\end{align}
(to be compared with the homogeneity property (\ref{eq:orbifold,PD})) imply \eqref{eq:13}.
The left hand side of \eqref{eq:14} vanishes because
$S^{A}(\phi_{i},\phi_{j})$ does not depend on the coordinates
$\underline{t}$. The equalities \eqref{eq:16} implies that the right
hand side also vanishes.
\end{proof}

From propositions \ref{prop:flat,connection} and \ref{prop:A,metric} we get

\begin{corollary}
  The tuple $(M_{A}, H^{A}, \nabla^{A}, S^{A}, n)$ is a quantum differential system on
  $\ppit^{1}\times M_{A}$.
\end{corollary}

\begin{definition}\label{defi:big,A,D,module} The quantum differential system $(M_{A}, H^{A}, \nabla^{A}, S^{A}, n)$ is
  called {\em the big $A$-model quantum differential system associated to}
  $\mathcal{X}$.
\end{definition}

\begin{remark}\label{AmodIrvsDM}
  Iritani defines also a $A$-model quantum differential system  (which he calls a
  ``$A$-model $D$-module'' \cite[definition 2.2]{Ir}, the distinction
  between these two terminologies will become clear later, see remark
  \ref{BmodIrvsDM}) and his definition is very similar to ours. There
  are some mild differences: the first one is that Iritani considers
  the opposite of our Higgs field and, in order to identify $H^{A}$
  with $\pi^{*}TM_{A}$, he uses $\phi_{i}\mapsto \partial_{i}$ whereas
  we use $\phi_{i}\mapsto -\partial_{i}$ (we choose the minus sign
  because usually the infinitesimal period map on the $B$-side is
  defined with a minus sign).  The second one is that Iritani
  considers the matrix $R_{\infty}-\frac{n}{2}\id$ which has symmetric
  eigenvalues with respect to $0$ (in our case, the eigenvalues are
  symmetric with respect to $n/2$). $\blacklozenge$
\end{remark}

\subsection{The small $A$-model quantum differential system} 

On a manifold $X$, the small quantum product is the restriction of the
big one to $H^{2}(X,\cit)$, that is $\bullet_{ \underline{t}}$ where
$\underline{t}\in H^{2}(X,\cit)$.  The classes in $H^{2}(X,\cit)$ play
a special role because they satisfy the divisor axiom for
Gromov-Witten invariants.  For orbifolds, the
divisor axiom works only for classes in the second cohomology group of
the untwisted sector (see Theorem 8.3.1 of \cite{AGV}), that is $H^{2}(\mathcal{X},\cit)$ (and not $H
^{2}_{\orb}(\mathcal{X},\cit)$).

\subsubsection{Restriction of the big $A$-model quantum differential system} 

We first restrict the big $A$-model quantum differential system $(M_{A}, H^{A},\nabla^{A}, S^{A}, n)$ to $M_{A}^{\smal}:=
H^{2}(\mathcal{X},\cit)$ and we get a quantum differential system on
$\ppit^{1}\times M_{A}^{\smal}$ denoted by 
$$(M_{A}^{\smal},H^{A,\smal},\nabla^{A,\smal},S^{A,\smal}, n).$$
 Let
$\underline{t}^{\smal}:=(t_{1}, \ldots ,t_{s})$ be the coordinates on
$M_{A}^{\smal}$. The restricted connection is
\begin{equation}\label{eq:connection}
  \nabla^{A,\smal} =d_{M_{A}^{\smal}}+d_{\ppit^{1}} - \frac{1}{z}\pi^{\ast}\Phi^{\smal} + \left(\frac{1}{z}\Phi^{\smal}(\mathfrak{E}^{\smal})+R_{\infty}\right)\frac{dz}{z}
\end{equation}
where $\Phi^{\smal}$ (resp. $\mathfrak{E}^{\smal}$ ) is the
restriction of $\Phi$ (resp.$\mathfrak{E}$) on $TM_{A}^{\smal}$. 
In coordinates, we have
$$
\Phi^{\smal}=\sum_{i=1}^{s}\Phi^{(i)}(\underline{t}^{\smal})
dt_{i}\mbox{ and }\mathfrak{E}^{\smal}=\sum_{i=1}^{s}
r_{i}\partial_{i}.
$$
Notice that $\mathfrak{E}^{\smal}$ is uniquely determined by
$c_{1}(T\mathcal{X})$ and that $\Phi^{\smal}(\mathfrak{E}^{\smal})$ is the
small quantum multiplication by $c_{1}(T\mathcal{X})$.

\subsubsection{An action of $\Pic(\mathcal{X})$} 
For manifolds, the quantum product is equivariant with respect to the
action of the Picard group. In this section, following Iritani
\cite{Ir}, we extend this action to the orbifold case.

Let $L$ be a line bundle on the orbifold
$\mathcal{X}$. For any point $x\in \mathcal{X}$, we have an action of
$Aut(x)$ on the fiber of $L$ at $x$ denoted by $L_{x}$ that is an
element on $GL(L_{x})$. Hence, for any point $(x,g)\in
\mathcal{X}_{v}\subset \mathcal{IX}$, we have an element $f_{v}(L)\in
\qit \cap [0,1[$ such that the action of $g$ on $L_{x}$ is the
multiplication by $e^{2\sqrt{-1}\pi f_{v}(L)}$.  The rational number
$f_{v}(L)$ depends only on $v\in T$ (see \cite[section 7]{AGV}).

\begin{remark}\label{rem:line,bundle,character}
  If $\mathcal{X}$ is a toric orbifold, then we have
  $\mathcal{X}=[Z/G]$ where $G:=\Hom(\Pic(\mathcal{X}),\cit^{\ast})$
  and $Z$ is a quasi-affine variety in some $\cit^{m}$ (cf. \cite{BCS}
  and \cite{FMN} for a more precise definition). The inertia stack is
  parametrized by a finite subset $T$ of $G$. A line bundle $L$ on
  $\mathcal{X}$ is given by a character $\chi_{L}$ of $G$ (see
  \cite{FMN}). In this special case, $f_{v}(L)$ is defined by the
  equality $\chi_{L}(v)=e^{2\pi\sqrt{-1}f_{v}(L)}$. $\blacklozenge$
\end{remark}

\noindent We define now an action of $\Pic(\mathcal{X})$ on $(M_{A}^{\smal}, H^{A,\smal},\nabla^{A,\smal},S^{A,\smal}, n) $ as follows:
\begin{enumerate}
\item on the fibers of $H^{A,\smal}$, for $\alpha\oplus \bigoplus_{v \in T}
  \alpha_{v}\in H^{*}(\mathcal{X},\cit)\oplus \bigoplus_{ v\in T}
  H^{*}(\mathcal{X}_{v},\cit)$ the action is given by
  \begin{align}\label{eq:action,fibre}
L\cdot \left(\alpha\oplus \bigoplus_{v \in T} \alpha_{v}\right)=\alpha\oplus
\bigoplus_{v\in T}e^{2\pi\sqrt{-1}f_{v}(L)}\alpha_{v}
  \end{align}
\item on $M_{A}^{\smal}= H^{2}(\mathcal{X},\cit)$ we define 
  \begin{align}
    \Pic(\mathcal{X})\times H^{2}(\mathcal{X},\cit)&\longrightarrow
    H^{2}(\mathcal{X},\cit) \label{eq:action,base}\\
    \left(L,\sum_{i=1}^{s}t_{i}\phi_{i}\right)&\longmapsto
    \left(\sum_{i=1}^{s}t_{i}\phi_{i}\right)-2\pi\sqrt{-1}c_{1}(L)=\sum_{i=1}^{s}(t_{i}-2\pi\sqrt{-1}L_{i})\phi_{i} \nonumber
    \end{align} where $c_{1}(L)=\sum_{i=1}^{s}L_{i}\phi_{i}$.
\end{enumerate}

\begin{proposition}[see proposition 2.3 of \cite{Ir}]\label{prop:action,classes} (1) The small quantum
  product is equivariant with respect to this action: for any classes
  $\alpha,\beta\in H^{*}_{\orb}(\mathcal{X},\cit)$, for any point
  $\underline{t}^{\smal}\in H^{2}(\mathcal{X},\cit)$ and for any $L \in
  \Pic(\mathcal{X})$, we have
\begin{displaymath} 
  (L\cdot\alpha)\bullet_{L\cdot \underline{t}^{\smal}}(L\cdot\beta)=L\cdot(\alpha\bullet_{\underline{t}^{\smal}}\beta).
\end{displaymath}
\noindent (2) The pairing $S^{A,\smal}(\cdot,\cdot)$ is invariant with respect to this action. 
\end{proposition}

\begin{proof}
  Recall that we denote by $\phi^{i}$ the Poincar\'e dual of
  $\phi_{i}$.  By definition of the small quantum product, we have
\begin{align*}
  (L\cdot\alpha)\bullet_{L\cdot
    \underline{t}^{\smal}}(L\cdot\beta)&=\sum_{d\in
    H_{2}(\mathcal{X},\qit)}\sum_{i=0}^{N}\langle L\cdot \alpha,
  L\cdot\beta,\phi_{i}\rangle_{0,3,d}
  \phi^{i}e^{\int_{d}(\underline{t}_{\smal}-2\pi\sqrt{-1}c_{1}(L))}.
\end{align*}
By definition of the Poincar\'e duality, we have that
$L\cdot \phi^{i}=L^{-1}\cdot \phi_{i}$. Using the proof of
Proposition 2.3 in \cite{Ir}, we deduce that
\begin{align*}
  (L\cdot\alpha)\bullet_{L\cdot \underline{t}^{\smal}}(L\cdot\beta)
  &=\sum_{d\in H_{2}(\mathcal{X},\qit)}\sum_{i=0}^{N}\langle L\cdot
  \alpha, L\cdot\beta,L\cdot\phi_{i}\rangle_{0,3,d}
  \left(L\cdot\phi^{i}\right)e^{\int_{d}(\underline{t}_{\smal}-2\pi\sqrt{-1}c_{1}(L))}\\
  &=\sum_{d\in H_{2}(\mathcal{X},\qit)}\sum_{i=0}^{N}\langle \alpha,
  \beta,\phi_{i}\rangle_{0,3,d}
  \left( L\cdot \phi^{i} \right)e^{\int_{d}\underline{t}_{\smal}}\\
&=L\cdot(\alpha\bullet_{\underline{t}_{\smal}\beta}).
\end{align*}

For the second statement, we show that for any $\alpha_{v}\in
H^{*}(\mathcal{X}_{v},\cit)$, for any $\alpha_{w}\in
H^{*}(\mathcal{X}_{w},\cit)$ and for any $L\in \Pic(\mathcal{X})$, we
have :
\begin{displaymath}
  S(L\cdot \alpha_{v}, L\cdot \alpha_{w})=S(\alpha_{v},\alpha_{w}).
\end{displaymath}
We have that $S(\alpha_{v},\alpha_{w})\neq 0$ implies that the
involution of $I\mathcal{X}$ sending $(x,g)\to (x,g^{-1})$ maps
$\mathcal{X}_{v}$ to $\mathcal{X}_{w}$ (see the definition of the
orbifold Poincar\'e duality in \cite{CR2}). This implies that
$f_{v}(L)+f_{w}(L)\in\{0,1\}$. Hence, we have 
\begin{displaymath}
  S(L\cdot \alpha_{v}, L\cdot \alpha_{w})=e^{2\pi\sqrt{-1}(f_{v}(L)+f_{w}(L))}S(\alpha_{v},\alpha_{w})=S(\alpha_{v},\alpha_{w}).
\end{displaymath}  
\end{proof}

\begin{remark}\label{rem:A,action,invariant,Picard}
  By the divisor axiom, the variables corresponding to
  $H^{2}(\mathcal{X},\cit)$ appear as exponential in the genus $0$ Gromov-Witten
  potential. For $i\in\{1, \ldots ,s\}$, we have indeed terms of the form
  $e^{t_{i}\int_{\beta}\phi_{i}}$ for $\beta\in
  H_{2}(\mathcal{X},\qit)$ and the action above acts on these terms as follows
\begin{equation}\label{eq:21}
  L\cdot e^{\sum_{i=1}^{s}t_{i}\int_{\beta}\phi_{i}}=e^{\sum_{i=1}^{s}t_{i}\int_{\beta}\phi_{i}}e^{-2\pi\sqrt{-1}\int_{\beta}c_{1}(L)}.
\end{equation}
Since, for orbifolds, the classes $\beta$ and the Chern classes are
rational, the action of the Picard group is not trivial.  So the
multiplication by $\exp\left(-2\pi\sqrt{-1}\int_{\beta}c_{1}(L)\right)$
has to be corrected by a natural action on the fibers of $H^{A,\smal}$
on the twisted cohomology
classes in order to get the proposition above.  For manifolds, the
homology class $\beta$ and the Chern classes are integral, hence the
action \eqref{eq:21} is trivial: the quantum product for
manifold is invariant with respect to this action. $\blacklozenge$
\end{remark}

\subsubsection{The quotient structure}
It follows from proposition \ref{prop:action,classes} that the quantum differential system
$(M_{A}^{\smal},H^{A,\smal},S^{A,\smal},n)$ is
$\Pic(\mathcal{X})$-equivariant. Hence, it defines a quotient quantum differential system
denoted by
$$\mathcal{S}^{A}:=(\mathcal{M}_{A}, \widetilde{H}^{A,\smal},\widetilde{\nabla}^{A,\smal},\widetilde{S}^{A,\smal},n)$$ 
where 
\begin{displaymath}
  \mathcal{M}_{A}:=H^{2}(\mathcal{X},\cit)/ \Pic(\mathcal{X})\simeq
  (\cit^{*})^{s}.
\end{displaymath}

\begin{corollary}
  The tuple  $\mathcal{S}^{A}$ is a quantum differential system on
  on $\ppit^{1}\times\mathcal{M}_{A}$.
\end{corollary}

\begin{definition}\label{AMSstructure}
  The quantum differential system $(\mathcal{M}_{A},
  \widetilde{H}^{A,\smal},\widetilde{\nabla}^{A,\smal},\widetilde{S}^{A,\smal},
  n)$ is called the {\em small $A$-model quantum differential system}. 
\end{definition}

\begin{remark}\label{rem:global,section,H/pic} 
  For $i\in\{0, \ldots ,N\}$ ($N+1$ is the dimension of the full orbifold cohomology), $\phi_{i}$ is a global section
  of $H^{A,\smal}$. We have 
    \begin{displaymath}
      \phi_{i} \mbox{ is a global section of } \widetilde{H}^{A,\smal}
      \Longleftrightarrow L\cdot \phi_{i}=\phi_{i},\  \forall L \in \Pic(\mathcal{X}).
    \end{displaymath}
    We deduce that the classes $\phi_{i}$ in the cohomology of the
    untwisted sector are global sections of $\widetilde{H}^{A,\smal}$.
    Notice that if $s_{1}$ and $s_{2}$ are global sections of
    $\widetilde{H}^{A,\smal}$, then so is
    $s_{1}\bullet_{\underline{t}^{\smal}}s_{2}$. To find a basis of
    global section of $\widetilde{H}^{A,\smal}$, we will look for quantum product of global sections i.e. $s_{1}\bullet_{\underline{t}^{\smal}}s_{2}$. $\blacklozenge$
\end{remark}

In the following, we define coordinates on $\mathcal{M}_{A}$ (which depend on a choice)  and then
we want to write the connection $\widetilde{\nabla}^{A,\smal}$ in
these coordinates (see Formula \eqref{eq:nable,tilde,variable,q}).

For $i\in\{1, \ldots ,s\}$, we put $q_{i}:=\exp(t_{i})$. However, the
$\underline{q}:=(q_{1}, \ldots ,q_{s})$ are not coordinates on
$\mathcal{M}_{A}$ because they are not $\Pic(\mathcal{X})$-invariant.
To be precise for any $L\in \Pic(\mathcal{X})$, we have
\begin{equation}\label{eq:action,q}
  L\cdot q_{i}=q_{i}e^{-2\pi\sqrt{-1}L_{i}}
\end{equation}
where $L_{i}$ are rational numbers\footnote{If the $L_{i}$ are
  integers then the $q$'s are $\Pic(\mathcal{X})$-invariant i.e. they
  are coordinates on $\mathcal{M}_{A}$.} defined by
$c_{1}(L)=\sum_{i=1}^{s}L_{i}\phi_{i}\in H^{2}(\mathcal{X},\qit)$, see
\eqref{eq:action,base}.  However if we choose $\mathcal{L}_{1}, \ldots
,\mathcal{L}_{s}$ as generators of of
$\Pic(\mathcal{X})/\tor(\Pic(\mathcal{X}))$\footnote{Observe that the
  first Chern class of a torsion line bundle vanishes.} and put
$\phi_{i}:=c_{1}(\mathcal{L}_{i})$, then the $L_{i}$'s are now
integers {\em i.e,} $(q_{1}, \ldots ,q_{s})$ are
coordinates\footnote{For manifolds, the situation is easier because
  one can choose $\phi_{i}$ as an integral cohomology class. Since
  $c_{1}(L)$ is an integral cohomology class, the $L_{i}$'s are
  integers.} on $\mathcal{M}_{A}$.  In such a choice of coordinates on
$\mathcal{M}_{A}$, the connection $\widetilde{\nabla}^{A, \smal}$ is
given by
\begin{equation}
  \label{eq:nable,tilde,variable,q}
  \widetilde{\nabla}^{A,\smal} =d_{\mathcal{M}_{A}}+d_{\ppit^{1}} -
  \frac{1}{z}\widetilde{\Phi}^{\smal} +
  \left(\frac{1}{z}\widetilde{\Phi}^{\smal}(\widetilde{\mathfrak{E}}^{\smal})+R_{\infty}\right)\frac{dz}{z}
\end{equation}
where
$$\widetilde{\Phi}^{\smal}=\sum_{i=1}^{s}\Phi^{(i)}
\frac{dq_{i}}{q_{i}} \mbox{ and } \widetilde{\mathfrak{E}}^{\smal}=\sum_{i=1}^{s}
r_{i}q_{i}\frac{\partial}{\partial q_{i}}.$$

\begin{remark}\label{rem:Iritani,global action}
  We first restrict the big $A$-model quantum differential system $(M_{A},
  H^{A},\nabla^{A}, S^{A},n)$ to $\ppit^{1}\times H^{2}(\mathcal{X},\cit)$ and
  then we quotient it by the action of $\Pic(\mathcal{X})$. 
  In \cite{Ir}, Iritani defines a global action, called Galois action,
  of $\Pic(\mathcal{X})$ on $(M_{A}, H^{A}, \nabla^{A},S^{A},n)$, giving a quantum differential system
  on $M_{A}/\Pic(\mathcal{X})$.  If we restrict it to
  $\mathcal{M}_{A}=H^{2}(\mathcal{X},\cit)/\Pic(\mathcal{X})$ we get
  the small $A$-model quantum differential system above.  $\blacklozenge$
\end{remark}

\subsection{Combinatorics}  

\label{combinatoire}
In order to describe the small $A$-model quantum differential system and its mirror, we introduce some combinatorics.\\

Let $w_{0},w_{1},\cdots ,w_{n}$ be positive integers.
Put $\mu:=w_{1}+\cdots+w_{n}$ (we use the letter $\mu$ because this will be the Milnor number on the B-side).
Denote by  
$$F:=\left\{\frac{\ell}{w_{i}}|\, 0\leq\ell\leq w_{i}-1,\ 0\leq i\leq n\right\}.$$
We denote by $f_{1},\cdots , f_{k}$ the elements of $F$ arranged
in increasing order:
$$0=f_{1}<f_{2}<\cdots<f_{k}<f_{k+1}:=1.$$
For $f\in\qit$, we define 
\begin{align}\label{eq:mi,Sf}
  S_{f}:=\{j|\ w_{j}f\in\zit\}\subset \{0,\cdots ,n\}\mbox{ and }
 m_{i}:=\prod_{j\in S_{f_{i}}}w_{j}.
\end{align}
The \textit{multiplicity}, denoted by $d_{i}$, of $f_{i}$ is the
positive integer defined by $d_{i}:=\# S_{f_{i}}.$
In particular we have $S_{f_{1}}=\{0,\cdots ,n\}$, $m_{1}=w_{0}\cdots w_{n}$ and $d_{1}=n+1$. Notice that
$$d_{1}+\cdots +d_{k}=\mu.$$
Let $c_{0},c_{1},\cdots , c_{\mu -1}$ be the sequence
$$\underbrace{f_{1},\cdots ,f_{1}}_{d_{1}},\underbrace{f_{2},\cdots ,f_{2}}_{d_{2}},\cdots ,\underbrace{f_{k},\cdots ,f_{k}}_{d_{k}}$$
arranged in increasing order.  It can be obtained as follows (see
\cite[p. 3]{DoSa2}): define inductively the sequence $(a(k),
i(k))\in\nit^{n+1}\times \{0,\cdots ,n\}$ by $a(0)=(0,\cdots ,0)$ ,
$i(0)=0$ and
$$a(k+1)=a(k)+ \mbox{\bf{1}}_{i(k)} \mbox{ where } i(k):=\min \{i|
a(k)_{i}/w_{i}=\min_{j} a(k)_{j}/w_{j}\}.$$ In particular,
$a(1)=(1,0,\cdots ,0)$, $a(n+1)=(1,\cdots ,1)$, $a(\mu )=(1,
w_{1},\cdots , w_{n})$ and $\sum_{i=0}^{n}a(k)_{i}=k$.  Then we have :
$$c_{k}=a(k)_{i(k)}/w_{i(k)}.$$

\begin{lemma} \label{lessk} We have $c_{0}=\cdots =c_{n}=0$, $c_{n+1}=\frac{1}{max_{i}w_{i}}$
and
$c_{k}+c_{\mu +n-k}=1$
for $k\geq n+1$.
\end{lemma}
\begin{proof}
See \cite[p. 2]{DoSa2}. 
\end{proof}

Define now, for $k=0,\cdots ,\mu -1$, $\alpha_{k}:=k-\mu c_{k}$. 

\begin{corollary} \label{lesalphak} 
We have $\alpha_{0}=0, \cdots , \alpha_{n}=n$, $\alpha_{k+1}\leq \alpha_{k}+1$ for all $k$,
$$\alpha_{k}+\alpha_{\mu +n-k}=n$$
for $k=n+1, \cdots ,\mu -1$ and
$$\alpha_{k}+\alpha_{n-k}=n$$
for $k=0,\cdots ,n$.
\end{corollary}

\noindent The $\alpha_{k}$'s will give the {\em spectrum at infinity}
of a certain regular function on the B-side (see section \ref{Bmodel})
and half of the {\em orbifold degree} on the A-side see Proposition \ref{lem:coates}.  Notice that
these numbers are integers if and only if $w_{i}|\mu$ for $i=0,\cdots
,n$.

\begin{example}\label{example}
  Let $w_{0}=1$, $w_{1}=2$,
  $w_{2}=2$. We have : 
\begin{itemize}
\item $\mu =5$, 
\item $f_{1}=0$, $d_{1}=3$,
  $f_{2}=\frac{1}{2}$, $d_{2}=2$, $S_{f_{1}}=\{0,1,2\}$ and
  $S_{f_{2}}=\{1,2\}$,
\item  $a(0)=(0,0,0)$, $a(1)=(1,0,0)$, $a(2)=(1,1,0)$, $a(3)=(1,1,1)$ , $a(4)=(1,2,1)$
\item $c_{0}=c_{1}=c_{2}=0$, $c_{3}=c_{4}=\frac{1}{2}$  and $\alpha_{0}=0,\ \alpha_{1}=1,\ \alpha_{2}=2,\ \alpha_{3}=\frac{1}{2},\  \alpha_{4}=\frac{3}{2}.$
\end{itemize} 
We will follow this example all along this paper. 
\end{example}

\subsection{The small $A$-model quantum differential system for weighted projective
  spaces}

We describe in this section the small $A$-model quantum differential system
$$\mathcal{S}^{A}_{w}=(\mathcal{M}_{A}, \widetilde{H}^{A,\smal},\widetilde{\nabla}^{A,\smal}, \widetilde{S}^{A,\smal},n)$$
associated with the weighted projective space $\ppit(w):=\ppit (w_{0}, \ldots ,w_{n})$, where $w_{0},\cdots , w_{n}$ are positive integers with $w_{0}=1$. 
The index $_{w}$ recalls these weights.

\subsubsection{The toric description}\label{sec:toric-description}

We use here the notations and the definitions given in section
\ref{combinatoire}. Recall that we assume $w_{0}=1$.  We follow the
definition of \cite{Coa} for weighted projective spaces, that is with
negative weights\footnote{In this paper, we use negative weights as
  \cite{Coa} because the mirror formula are easier for negative
  weights namely in \eqref{eq:19} of \S \ref{sec:correspondance} we will have $P^{\bullet j} \mapsto \omega_{j}$. In \cite[\S
  6.c]{Man}, the second author took positive weights and the
  correspondence was a bit more tricky.},
\begin{align}\label{defi,wps}
  \ppit(w_{0},w_{1}, \ldots ,w_{n}):=[\cit^{n+1}-\{0\}/\cit^{\ast}]
\end{align}
where the action is given by $\lambda(x_{0}, \ldots
,x_{n}):=(\lambda^{-w_{0}}x_{0}, \ldots ,\lambda^{-w_{n}}x_{n})$.

It is a toric Deligne-Mumford stacks in the sense of \cite{FMN} and \cite{BCS}. 
Its stacky fan is given by
\begin{itemize}
\item the lattice $N:=\zit^{n}$.
\item the morphism $\beta:\zit^{n+1}\to N$ that sends the canonical
  basis $e_{i}$ to $(0, \ldots
  ,0,1,0, \ldots ,0)$  and $e_{0}$ to $(-w_{1}, \ldots ,-w_{n})$.
\item the fan $\Sigma$ in $N$ is the complete fan where the rays are
  generated by $\beta(e_{i})$. 
\end{itemize}

\begin{remark}\label{rem:toric}
  (1) The Picard group of $\ppit (w)$ is $\zit$ and it is generated by the line bundle $\mathcal{O}(1)$.\\
  (2) \label{rem:toric,2}For $i\in\{0, \ldots ,n\}$, each $\beta(e_{i})$
  corresponds to a toric divisor $D_{i}$. This toric divisor is simply the
  canonical inclusion of $\ppit(w_{0}, \ldots ,\widehat{w}_{i}, \ldots
  ,w_{n})\hookrightarrow \ppit(w)$.  The line bundle associated to the toric
  divisor $D_{i}$ is $\mathcal{O}(w_{i})$. The situation when $w_{0}=1$ is
  particularly nice, because the toric divisor $D_{0}$ is $\mathcal{O}(1)$ which
  generates the Picard group.  We denote by $P:=c_{1}(\mathcal{O}(1))\in
  H^{2}(\ppit(w),\qit)\subset H^{2}_{\orb}(\ppit(w),\cit)$.
$\blacklozenge$
\end{remark}

\noindent For any subset $I=\{i_{1}, \ldots ,i_{\ell}\}\subset \{0, \ldots
,n\}$, we put $\ppit(w_{I}):=\ppit(w_{i_{1}}, \ldots ,w_{i_{\ell}})$. Recall the
sets $F$ and $S_{f}$ defined in \eqref{eq:mi,Sf}.  Following \cite{Man} and
\cite{Coa}, the inertia stack is
\begin{displaymath}
  \mathcal{I}\ppit(w):=\bigsqcup_{f\in F}\ppit(w_{S_{f}})
\end{displaymath}
For any $f\in F$, denote by $\mathbf{1}_{f}$ the image of the cohomology class
$\mathbf{1}\in H^{0}(\ppit(w_{S_{f}}),\cit)$ in $H^{*}_{\orb}(\ppit(w),\cit)$.
A basis of the orbifold cohomology $H^{*}_{\orb}(\ppit(w), \cit )$, which is a
$\cit$-vector space of dimension $\mu$, is given by the elements
\begin{equation}
  \label{eq:basis}
  \mathbf{1}_{f_{i}}P^{j}:=1_{f_{i}}\cup_{\orb}\stackrel{j- \mbox{times}}{\overbrace{P\cup_{\orb}\cdots\cup_{\orb}P}},\mbox{ for }\ i\in \{1, \cdots ,k\} \mbox{
    and } j\in\{0,\cdots ,d_{i}-1\}.
\end{equation}

 The orbifold degree is now
defined by
$$\deg^{\orb} \mathbf{1}_{f_{i}}P^{j}:= 2j + 2\sum_{k=0}^{n} \{-w_{k}f_{i}\}$$
where $\{r\}:=r-\lfloor r\rfloor$ is the fractional part of $r$. The
orbifold Poincar\'e duality (see \cite{Man}) is given by
\begin{equation}\label{eq:18}
  \langle \mathbf{1}_{f_{i}}P^{k}, \mathbf{1}_{f_{j}}P^{\ell}\rangle=
    \begin{cases}
      1/m_{i} & \mbox{ if  } f_{i}+f_{j}\in \nit \mbox{ and } k+\ell=d_{i}-1\\
      0 & \mbox{ otherwise}
     \end{cases}
\end{equation}
where $m_{i}=\prod_{j\in S_{f_{i}}}w_{j}$ (see \eqref{eq:mi,Sf}). Notice that if $f_{i}+f_{j}\in \nit$ then
$S_{f_{i}}=S_{f_{j}}$ so that the right hand side of \eqref{eq:18} is symmetric
in $i$ and $j$.

\subsubsection{ Description of the small $A$-model quantum differential system}
\label{quantumWPS}

Let $t_{1}$ be the coordinate on $H^{2}(\ppit(w),\cit)$,
$q:=\exp(t_{1})$ and $C^{\orb}(q)$ be the matrix of the endomorphism
$P\bullet_{q}$ of $H^{*}_{\orb}(\ppit(w), \cit )$ in the basis
$(\mathbf{1}_{f_{i}}P^{j})$. This matrix is computed in \cite{Coa}
(see also \cite{Guest}): we have
\begin{displaymath}
  C^{\orb}(q):=\left(
  \begin{matrix}
    0&0&0&\cdots&0& a_{\mu}q^{1-c_{\mu-1}}\\
    a_{1}q^{c_{1}-c_{0}}&0&0&\cdots&0&0 \\
    0&a_{2}q^{c_{2}-c_{1}}&0&&\vdots&\vdots \\
    \vdots&\ddots&\ddots&\ddots&\vdots& \vdots\\
    \vdots&&\ddots&\ddots&0&\vdots \\
    0&\cdots&\cdots&0&a_{\mu-1}q^{c_{\mu-1}-c_{\mu-2}}& 0\\
  \end{matrix}\right)
\end{displaymath}
where
 \begin{equation}\label{eq:20}
  a_{i}:= \begin{cases}
   1/m_{j} & \mbox{ if } i=d_{1}+\cdots+d_{j}\\
 1 & \mbox{ otherwise.}  
   \end{cases}
 \end{equation}
Following the remark \ref{rem:global,section,H/pic}, we define, for $i\in\{0,\cdots ,\mu -1\}$,
 $$(P^{\bullet_{q}})^{i}:=\underbrace{P\bullet_{q}\cdots\bullet_{q}P}_{i\ \mbox{times}} \mbox{ with } (P^{\bullet_{q}})^{0}:=\mathbf{1}_{f_{1}}.$$

\begin{lemma}\label{lem:coates}[See \cite{Coa}]
 (1) We have 
\begin{equation}\label{eq:17}
  (P^{\bullet_{q}})^{i}=q^{c_{i}}s_{i}\mathbf{1}_{c_{i}}P^{r(i)}
\end{equation}
where $r(i):=\#\{k\mid k<i \mbox{ and } c_{k}=c_{i}\}$ and $s_{i}=\prod_{k=0}^{n}w_{k}^{-\lceil c_{i}w_{k}\rceil}$.
In particular, for each $q\neq 0$, the  cohomology classes
$((P^{\bullet_{q}})^{i})_{0\leq i\leq\mu -1}$ form a basis 
of the vector space $H^{*}_{\orb}(\ppit(w),\cit)$.\\ 
(2) For every $i$, $\deg^{\orb}(P^{\bullet_{q}})^{i}=\deg^{\orb}\mathbf{1}_{c_{i}}P^{r(i)}=2\alpha_{i}$ (c.f. \S\ref{combinatoire} for the definition of $\alpha$'s). 
\end{lemma}

\begin{proof}
  The only part of the proof that is not in \cite{Coa} is that
  $\deg^{\orb}\mathbf{1}_{c_{i}}P^{r(i)}=2\alpha_{i}$.
  \begin{align*}
    \frac{1}{2}\deg^{\orb}\mathbf{1}_{c_{i}}P^{r(i)}&= 
\sum_{j=0}^{n}\{-c_{i}w_{j}\}+r(i)=-\sum_{j=0}^{n}\{c_{i}w_{j}\}+n+1-d_{j}+r(i)\\
    &=-c_{i}\mu+\sum_{j=0}^{n}\lfloor c_{i}w_{j}\rfloor+n+1-d_{i}+r(i)
    =-c_{i}\mu+d_{1}+\cdots+d_{i-1}+r(i)\\ 
    &=-c_{i}\mu+i=\alpha_{i}
  \end{align*}
  
\end{proof}

The following proposition refines the
remark \ref{rem:global,section,H/pic} for weighted projective spaces.

\begin{proposition}\label{prop:action,bases,P,1f}
  The Picard group $\Pic(\ppit(w))$ acts on the two basis
  $(\mathbf{1}_{f_{i}}P^{j})$ and $((P^{\bullet_{q}})^{i})$ of $H^{*}_{\orb}(\ppit(w))$ via the following
  formulas: 
\begin{align*}
  \mathcal{O}(d)\cdot \mathbf{1}_{f}P^{k} =
  e^{-2\pi\sqrt{-1}df}\mathbf{1}_{f}P^{k} &\mbox{ and }
  \mathcal{O}(d)\cdot
(P^{\bullet_{q}})^{i} =(P^{\bullet_{\mathcal{O}(d)\cdot q}})^{i}.
\end{align*}
\noindent for any $d\in \zit$. For $r\in \qit$, we have also $\mathcal{O}(d)\cdot q^{r} = q^{r}e^{-2\pi\sqrt{-1}dr}$.
\end{proposition}

\begin{proof}
  Because we take the definition of weighted projective spaces with
  negative weights (see Formula \eqref{defi,wps}), the line bundle
  $\mathcal{O}(d)$ corresponds to the character $\chi:\cit^{\ast}\to
  \cit^{\ast}$ which sends $z\to z^{-d}$. Using remark
  \ref{rem:line,bundle,character}, the action of $\mathcal{O}(d)$ on
  $\mathbf{1}_{f}P^{k}$ follows from the definition of the action (see
  formula \eqref{eq:action,fibre}). For the action on $q$, it follows
  from the definition (see formula \eqref{eq:action,base} and
  \eqref{eq:action,q}).  The action on $(P^{\bullet_{q}})^{i}$ follows
  from proposition \ref{prop:action,classes}.
\end{proof}

\begin{remark}\label{rem:global,sections}From \eqref{eq:17}, we put $s(q):=(P^{\bullet_{q}})^{i}=q^{c_{i}}s_{i}\mathbf{1}_{c_{i}}P^{r(i)}$. We have 
  \begin{align*}
    s(\mathcal{O}(d)\cdot q)&=(\mathcal{O}(d)\cdot
    q^{c_{i}})s_{i}\mathbf{1}_{c_{i}}P^{r(i)}\\
&= q^{c_{i}}e^{-2\pi\sqrt{-1}dc_{i}}s_{i}\mathbf{1}_{c_{i}}P^{r(i)}\\
&=q^{c_{i}}s_{i}\left(\mathcal{O}(d)\cdot
  \mathbf{1}_{c_{i}}P^{r(i)}\right) \\
&=\mathcal{O}(d)\cdot s(q).
  \end{align*}
  As expected from remark \ref{rem:global,section,H/pic}, for
  $i\in\{0, \ldots ,N\}$, the section $(P^{\bullet_{q}})^{i}$ is a
  $\Pic(\ppit(w))$-equivariant section, hence it induces a global
  section of the bundle $\widetilde{H}^{A,\smal}$. $\blacklozenge$
\end{remark}

As shown by the previous proposition, we prefer the basis
$((P^{\bullet_{q}})^{i})$ because  it provides a basis of global
sections of the small $A$-model quantum differential system. We first compute the
pairing $\widetilde{S}^{A,\smal}(\cdot,\cdot)$ in this basis.

\begin{proposition}\label{prop:pairingglobal} 
  The pairing $\widetilde{S}^{A,\smal}(\cdot,\cdot)$ in the basis
  $((P^{\bullet_{q}})^{i})$ is
\begin{displaymath}
  \widetilde{S}^{A,\smal} \left((P^{\bullet_{q}})^{i}, (P^{\bullet_{q}})^{j}\right)=
  \begin{cases}
    z^{n}m_{1}^{-1} & \mbox{if }i+j=n \\
    z^{n}m_{1}^{-1}q w^{-w} & \mbox{ if } i+j=n+\mu\\ 
 0 & \mbox{otherwise}
  \end{cases}
\end{displaymath}
where $w^{-w}:=\prod_{i=0}^{n}w_{i}^{-w_{i}}$.
\end{proposition}
\begin{proof}
  Recall that $\widetilde{S}^{A,\smal}(\cdot,\cdot):=z^{n}\langle \cdot,
  \cdot \rangle$.  We will use the formulas \eqref{eq:18} and
  \eqref{eq:17}.  The first case follows from the
  equivalence between $i+j=n$ and $c_{i}=c_{j}=0$.  From
  \cite[Proposition 6.1.(3)]{Man}, we have that $i+j=n+\mu$ is
  equivalent to $c_{i}+c_{j}=1$ and $r(i)+r(j)=d_{i}-1$. We conclude
  using the fact that $ s_{i}s_{j}=w^{-w}\prod_{k\notin S_{c_{i}}}w_{k}^{-1}$
if $c_{i}+c_{j}=1$.
\end{proof}

\begin{remark} Notice that if $w_{0}=\cdots=w_{n}=1$ 
the bases $((P^{\bullet_{q}})^{i})_{0\leq i\leq n}$ and
$(\mathbf{1}_{f_{i}}P^{j})$ are equal and that the pairing does not depend on
$q$. $\blacklozenge$
\end{remark}

Put
\begin{align*}
 A_{\infty}&
 :=\frac{1}{2}\Diag(\deg^{\orb}1,\deg^{\orb}P, \ldots
 ,\deg^{\orb}(P^{\bullet_{q}})^{\mu-1})=\Diag(\alpha_{0}, \ldots ,\alpha_{\mu-1})
\end{align*}

\noindent The following proposition completes the description of the small $A$-model quantum differential system $\cal{S}^{A}_{w}$.

\begin{proposition} 
\label{prop:nabla,basis,nonflat}
(1) The matrix of the connection $\widetilde{\nabla}^{A,\smal}$ in the basis $(\mathbf{1}_{f_{i}}P^{j})$ is
 \begin{equation}\label{eq:flat}
 -\frac{1}{z}C^{\orb}(q)\frac{dq}{q} + \left(\frac{1}{z}\mu
     C^{\orb}(q)+A_{\infty}\right)\frac{dz}{z}
\end{equation}
\noindent (2) The matrix of the connection $\widetilde{\nabla}^{A,\smal}$ in the
basis $((P^{\bullet_{q}})^{i})$ is
\begin{equation*}
  \left(-\frac{\mu C(q)}{ z}-A_{\infty}+H \right)\frac{d q}{\mu q} +
  \left(\frac{\mu C(q)}{z}+A_{\infty}\right)\frac{dz}{z}
\end{equation*}
where $H:=\Diag(0, \ldots ,\mu-1)$ and
$$C(q)=\left ( \begin{array}{cccccc}
0   & 0   & 0 & \cdots & 0   & q/w^{w}\\
1   & 0   & 0 & \cdots & 0   & 0\\
0   & 1  & 0 & \cdots & 0   & 0\\
..  & ... & . & \cdots & .   & .\\
..  & ... & . & \cdots & .   & .\\
0   & 0   & . & \cdots & 1   & 0 
\end{array} \right ).$$
\end{proposition}

\begin{proof}
  (1) Since $c_{1}(T\ppit (w))=\mu P$ by \cite[lemma 3.21]{Man}, we
  have
$$\widetilde{\Phi}^{\smal}=(P\bullet_{q})\frac{dq}{q},\
\widetilde{\mathfrak{E}}^{\smal}=\mu P \mbox{ and }
\widetilde{\Phi}^{\smal}(\widetilde{\mathfrak{E}}^{\smal})= \mu
(P\bullet_{q}).$$ The proposition then follows from the definition of
$\widetilde{\nabla}^{A,\smal}$ (see equation \eqref{eq:nable,tilde,variable,q}). \\
(2) Follows now from a straightforward computation via the change of
basis \eqref{eq:17}.
\end{proof}

\begin{remark}\label{rem:A,side,flat}
(1) As we have seen in proposition \ref{prop:action,bases,P,1f}, the
  cohomology class $\mathbf{1}_{f_{i}}P^{j}$ does not define a
  global section of the small $A$-model quantum differential system, whereas
  $(P^{\bullet_{q}})^{i}$ does.
  This explains the fact that the matrix $C(q)$ (resp.
  $C(q)$) contains rational (resp. integer) powers of $q$.\\
(2) Another way to measure the difference between the bases
  $(\mathbf{1}_{f_{i}}P^{j})$ and $(P^{\bullet_{q}})^{i}$ is to
  consider the restriction $\bigtriangledown$ of
  $\widetilde{\nabla}^{\smal}$ to $\{\infty\}\times \mathcal{M}_{A}$.
  We have :
\begin{itemize}
\item  $\bigtriangledown (\mathbf{1}_{f_{i}}P^{j})=0$,
\item $\bigtriangledown (P^{\bullet_{q}})^{i}=R((P^{\bullet_{q}})^{i})\frac{dq}{q}$.
\end{itemize}
where  $R:=\mu^{-1}(-A_{\infty}+H)=\Diag(c_{0}, \ldots ,c_{\mu-1})$ is the residue matrix of $\bigtriangledown$ (see Corollary \ref{coro:Deligne,extension}). In other words, the basis $(\mathbf{1}_{f_{i}}P^{j})$ is
$\bigtriangledown$-flat whereas $((P^{\bullet_{q}})^{i})$ is not.  $\blacklozenge$
\end{remark}

 \begin{remark}\label{rem:limit}
   The matrix $C^{\orb}(0)$ is the matrix of the endomorphism $P\cup_{\orb}$ and does not generate the orbifold cohomology ring in general:
   from the matrix $C^{\orb}(0)$, we can not get all the orbifold products
   $1_{f_{i}}P^{j}\cup_{\orb}1_{f_{k}}P^{\ell}$.  $\blacklozenge$
 \end{remark}

 \begin{example}\label{exampleWPS} For $\ppit (1,2,2)$ we have
   $$C^{\orb}(q)=
     \left(\begin{matrix}
       0&0&0&0&\frac{1}{4}q^{1/2}\\
       1&0&0&0&0\\
       0&1&0&0&0\\
       0&0&\frac{1}{4}q^{1/2}&0&0\\
       0&0&0&1&0\\
     \end{matrix}\right)$$
\noindent In particular,
$$C^{\orb}(0)=
     \left(\begin{matrix}
       0&0&0&0&0\\
       1&0&0&0&0\\
       0&1&0&0&0\\
       0&0&0&0&0\\
       0&0&0&1&0\\
     \end{matrix}\right)$$
and we can not get the equality $\mathbf{1}_{1/2}\cup_{\orb}\mathbf{1}_{1/2}P=P^{2}$ (see example \ref{exampleprod} below) from $C(0)$.
 \end{example}

\section{$B$-model}
\label{Bmodel}

\subsection{The setting}
\label{setting}

Givental in \cite{Givental1},\cite{givental2} and Hori-Vafa in
\cite{Hori} have offered a mirror partner for toric manifolds and Iritani
\cite{Ir} has explained how to construct a mirror candidate for a toric
orbifold. We briefly recall this construction in the case of the weighted projective space $\ppit (1,w_{1},
\ldots ,w_{n})$. 

We start with the following exact sequence
$$0\longrightarrow \Pic(\ppit (w))\longrightarrow \zit^{n+1}\stackrel{\beta}{\longrightarrow}N\longrightarrow 0$$
where $\beta:\zit^{n+1}\to N$ is the map defined via the stacky fan (see section \ref{sec:toric-description}).
Applying the functor $\Hom_{\zit}(\cdot,\cit^{*})$, we get : 
$$1\longrightarrow (\cit^{*})^{n}\longrightarrow (\cit^{*})^{n+1}\stackrel{\pi}{\longrightarrow} \cit^{*}\longrightarrow 1$$
This gives our mirror candidate to $\ppit(w)$,

$$\xymatrix@1{
(\cit^{*})^{n+1} \ar[r]^-{\widetilde{F}} \ar[d]_\pi & \cit\\
\mathcal{M}_{B}:=\cit^{*}}$$

\noindent where $\widetilde{F}(u_{0}, \ldots ,u_{n})=\sum_{i=0}^{n} u_{i}$ and
$\pi(u_{0}, \ldots ,u_{n})=u_{0}u_{1}^{w_{1}}\cdots u_{n}^{w_{n}}$.  
Denote by
$x$ the coordinate on $\mathcal{M}_{B}$. As
all the fibers of $\pi$ are isomorphic
to the torus $U:=(\cit^{*})^{n}$, we can also consider 
$$F: U\times \mathcal{M}_{B} \longrightarrow \cit $$
defined by
\begin{equation}\label{eq:defi,F}
  F(u_{1}, \ldots ,u_{n},x)=u_{1}+\cdots+u_{n}+\frac{x}{u_{1}^{w_{1}}\cdots u_{n}^{w_{n}}}.
\end{equation}
which is a deformation of 
$f:U\rightarrow\cit$ defined by
$$f(u_{1},\cdots ,u_{n})=u_{1}+\cdots +u_{n}+\frac{1}{u_{1}^{w_{1}}\cdots u_{n}^{w_{n}}}.$$
We will write
$$u_{0}=\frac{1}{u_{1}^{w_{1}}\cdots u_{n}^{w_{n}}}.$$

\begin{remark} If we identify the monomial
  $\prod_{i=0}^{n}u_{i}^{a_{i}}$ with the point $(a_{0}, \ldots
  ,a_{n})\in \zit^{n+1}$, we see that each monomial $u_{i}$
  corresponds to the point $\beta(e_{i})\in N$ where $e_{i}$ is the
  canonical basis of $\zit^{n+1}$. We interpret $\beta(e_{i})$ as the
  toric divisor $D_{i}$ (see Remark \ref{rem:toric}).  In particular,
  the monomial $u_{0}$ corresponds to $D_{0}=\mathcal{O}(1)$ and we
  can expect that the multiplication by $u_{0}$ corresponds to the
  multiplication by $P:=c_{1}(\mathcal{O}(1))$: this will be shown in
  section \ref{mirrorsmall}.$\blacklozenge$
\end{remark}

\subsection{Gauss-Manin systems and Brieskorn lattices}

Let
$$G=\frac{\Omega^{n}(U)[x, x^{-1}, \tau ,\tau^{-1}]}{(d_{u}-\tau d_{u}F)\wedge\Omega^{n-1}(U)[x, x^{-1}, \tau ,\tau^{-1}]}$$
be the (Fourier-Laplace transform of the) Gauss-Manin system of $F$,
and
$$G_{0}=\frac{\Omega^{n}(U)[x, x^{-1}, \tau^{-1}]}{(\tau^{-1}d_{u}-d_{u}F)\wedge \Omega^{n-1}(U)[x, x^{-1}, \tau^{-1}]}$$
be (the Fourier-Laplace transform of) its Brieskorn lattice, where the notation $d_{u}$ means that the differential is taken with respect to the 
coordinates $u=(u_{1},\cdots ,u_{n})$ of $U$ only.
The $\cit[x,x^{-1},\tau,\tau^{-1}]$-module $G$ is equipped with a flat connection $\nabla^{B}$ defined by 
\begin{equation}\label{def:ConnGM}  
\nabla^{B}_{\partial_{\tau}}(\omega_{i}\tau^{i})=i\omega_{i}\tau^{i-1}-F\omega_{i}\tau^{i}\ \mbox{and} \ \nabla^{B}_{\partial_{x}}(\omega_{i}\tau^{i})=\mathcal{L}_{\partial_{x}}(\omega_{i})\tau^{i}-\frac{\partial F}{\partial x}\omega_{i}\tau^{i+1}
\end{equation}
where $\mathcal{L}$ denotes the Lie derivative.
Assume moreover that $G_{0}$ is free over $\cit [x,x^{-1},\tau ^{-1}]$. We will say that a basis $\omega$ of $G_{0}$ over $\cit [x,x^{-1},\tau ^{-1}]$ is a solution of the {\em Birkhoff problem} for $G_{0}$ if the matrix of $\nabla^{B}$ in the basis $\omega$ is
$$(A(x)\tau +B(x))\frac{d\tau}{\tau}+(C(x)\tau +D(x))dx$$
where $A(x)$, $B(x)$, $C(x)$ and $D(x)$ are matrices with coefficients in $\cit [x,x^{-1}]$ (see for instance \cite[Chapitre VI.2]{Sab1}).

The Gauss-Manin system of $f$ and its Brieskorn lattice are respectively defined by
$$G^{o}=\frac{\Omega^{n}(U)[\tau ,\tau^{-1}]}{(d-\tau df)\wedge\Omega^{n-1}(U)[\tau ,\tau^{-1}]}$$
and
$$G_{0}^{o}=\frac{\Omega^{n}(U)[\tau^{-1}]}{(\tau^{-1}d-df)\wedge \Omega^{n-1}(U)[\tau^{-1}]}.$$
$G^{o}$ is also equipped with a flat connection $\nabla^{B,o}$ defined by
$$\nabla^{B,o}_{\partial_{\tau}}(\omega_{i}\tau^{i})=i\omega_{i}\tau^{i-1}-f\omega_{i}\tau^{i}$$
\noindent (see for instance \cite[Section 2]{DoSa1}). There is of course a Birkhoff problem for $G_{0}^{o}$: a solution will be a 
basis $\omega^{o}$ of
$G_{0}^{o}$ over $\cit [\tau ^{-1}]$ in which the matrix of $\nabla^{B,o}$ is
$(A^{o}\tau +B^{o})\frac{d\tau}{\tau}$ where $A^{o}$ and $B^{o}$ are two constant matrices (and we assume here that $G_{0}^{o}$ is free of finite rank on $\cit [\tau^{-1}]$).

\subsection{A $B$-model quantum differential system}
\label{deformation}

We look for a trivial bundle on $\ppit^{1}\times \mathcal{M}_{B}$, equipped with a connection and a flat pairing, isomorphic to the one considered in section \ref{Amodel}. 
In general, a solution of the Birkhoff problem for the Brieskorn lattice $G_{0}$ yields such objects. However, such a solution is not unique and, on this side, we have to take care of some choices: for instance, two different solutions could produce two residue matrices along $\tau =0$ (the matrix $B(x)$ with the notations above) which are not conjugate. 
This has motivated the definition of {\em canonical} solutions in \cite{DoSa1}, given by Hodge theory using M. Saito's method (see \cite[Section 5]{DoSa2} for a precise description in our setting).
It should be emphasized that the best solution in our context, {\em i.e} the one which fits mirror symmetry (see theorem \ref{quantum}), is closely related with the canonical solutions of the Birkhoff problem for $G^{o}$ given in \cite{DoSa2} (see remark \ref{start} (1) below).

\subsubsection{A trivial bundle}
\label{canonicalbasis}

 Let
$$\Gamma_{0}=\{(y_{1},\cdots ,y_{n})\in\rit^{n} |  y_{1}+\cdots +y_{n}=1\}$$
and 
$$\chi_{0}=u_{1}\frac{\partial}{\partial u_{1}}+\cdots +u_{n}\frac{\partial}{\partial u_{n}},$$
$$\Gamma_{j}= \left\{(y_{1},\cdots ,y_{n})\in\rit^{n} |  y_{1}+\cdots +y_{j-1}+\left(1-\frac{\mu}{w_{j}}\right)y_{j}+\cdots +y_{n}=1\right\}$$
and
$$\chi_{j}= u_{1}\frac{\partial}{\partial u_{1}}+\cdots +u_{j-1}\frac{\partial}{\partial u_{j-1}}+\left(1-\frac{\mu}{w_{j}}\right)u_{j}\frac{\partial}{\partial u_{j}}+\cdots +u_{n}\frac{\partial}{\partial u_{n}}$$
for $j=1,\cdots ,n$. 
The $\Gamma_{j}$'s are the faces of dimension $n-1$ of the Newton polyhedron of $f$ at infinity (see \cite{Ko}). 
We define, for $j=0,\cdots ,n$,
$$h_{j}=\chi_{j}(F)-F.$$
We thus have $h_{0}=-\mu x u_{0}$ and $h_{j}=-\frac{\mu}{w_{j}}u_{j}$ if $j=1,\cdots ,n$. 
 Last we put, for $g=u_{1}^{r_{1}}\cdots u_{n}^{r_{n}}$,
$$\phi_{0}(g)=r_{1}+\cdots +r_{n}$$ 
and, for $j=1,\cdots ,n$,
$$\phi_{j}(g)=r_{1}\cdots +r_{j-1}+\left(1-\frac{\mu}{w_{j}}\right)r_{j}+\cdots +r_{n}.$$
We will write $\partial_{\tau}$ instead of $\nabla^{B}_{\partial_{\tau}}$ for short.

\begin{lemma} \label{lemmebase} Let $\omega_{0}$ be the class of $\frac{du_{1}}{u_{1}}\wedge\cdots\wedge\frac{du_{n}}{u_{n}}$ in $G$. One has, for any monomial $g$, the equality
 $$(\tau\partial_{\tau}+\phi_{j}(g))g\omega_{0}=\tau h_{j}g\omega_{0}$$
in $G$, where $g\omega_{0}$ denotes the class of $g\frac{du_{1}}{u_{1}}\wedge\cdots\wedge\frac{du_{n}}{u_{n}}$ in $G$. In particular, 
$\tau\partial_{\tau}\omega_{0}=\tau h_{0}\omega_{0}.$
\end{lemma}
\begin{proof} This formula follows from the definition of $\partial_{\tau}$ (see equation (\ref{def:ConnGM})).
\end{proof}

\noindent This lemma is the starting point in order to solve the Birkhoff problem for $G_{0}$, as it has been the starting point to solve the one for $G_{0}^{o}$ 
in \cite[section 3]{DoSa2}. Set $\omega_{1}:=xu_{0}\omega_{0}$: then
$$-\frac{1}{\mu}\tau\partial_{\tau}\omega_{0}=\tau\omega_{1}$$
because $\tau\partial_{\tau}\omega_{0}=\tau h_{0}\omega_{0}$. One can iterate the process.
Recall the rational numbers $\alpha_{k}$ and the multi-indices $a(k)=(a(k)_{0},a(k)_{1},\cdots ,a(k)_{n})\in\nit^{n+1}$ defined in section \ref{combinatoire} (notice that $a(k)_{0}=1$ for $k\geq 1$ because $w_{0}=1$). 

\begin{lemma}\label{lemmeBirkhoff}
Let $$\omega_{k}=\frac{x}{w_{1}^{a(k)_{1}}\cdots w_{n}^{a(k)_{n}}}u_{0}u_{1}^{a(k)_{1}}\cdots u_{n}^{a(k)_{n}}\omega_{0}$$
for $k=1,\cdots ,\mu -1$. Then we have, in $G$,
$$-\frac{1}{\mu}(\tau\partial_{\tau}+\alpha_{k})\omega_{k}=\tau\omega_{k+1}$$
for $k=0,\cdots ,\mu -2$ 
and 
$$-\frac{1}{\mu}(\tau\partial_{\tau}+\alpha_{\mu -1})\omega_{\mu -1}=\frac{x}{w_{1}^{w_{1}}\cdots w_{n}^{w_{n}}}\tau\omega_{0}.$$
\end{lemma}
\begin{proof}
This is done as in \cite[section 2 and proof of proposition 3.2]{DoSa2}, using lemma \ref{lemmebase}.
\end{proof}
\noindent We will put $u^{a(k)}=u_{0}u_{1}^{a(k)_{1}}\cdots u_{n}^{a(k)_{n}}$: for instance, $u^{a(1)}=u_{0}$ and $u^{a(\mu )}=1$ because $u_{0}$ is defined by the equation $u_{0}u_{1}^{w_{1}}\cdots u_{n}^{w_{n}}=1$.\\

Let
$$A_{\infty}=\Diag (\alpha_{0},\cdots ,\alpha_{\mu -1}),$$
and, for $x\in \mathcal{M}_{B}$,
$$A_{0}(x)=\left ( \begin{array}{cccccc}
0   & 0   & 0 & \cdots & 0   & \mu x /w^{w}\\
\mu   & 0   & 0 & \cdots & 0   & 0\\
0   & \mu  & 0 & \cdots & 0   & 0\\
..  & ... & . & \cdots & .   & .\\
..  & ... & . & \cdots & .   & .\\
0   & 0   & . & \cdots & \mu   & 0 
\end{array} \right )$$

\noindent where $w^{w}=w_{1}^{w_{1}}\cdots w_{n}^{w_{n}}$.
We will preferably express our results in the variable $\theta :=\tau^{-1}$, also denoted on the $A$-side by $z$.

\begin{theorem}
  \label{basevarphi} The classes $\omega_{0},\cdots
  ,\omega_{\mu -1}$ form a basis $\omega$ of
  $G_{0}$ over $\cit [x,x^{-1},\theta]$. In this basis, 
  the connection $\nabla^{B}$ is
$$\left(-\frac{A_{0}(x)}{\theta} -A_{\infty}+H\right)\frac{dx}{\mu x}+\left(\frac{A_{0}(x)}{\theta} +A_{\infty}\right)\frac{d\theta}{\theta}$$ 
where $H=\Diag (0,1,\cdots ,\mu -1)$.
\end{theorem}
\begin{proof}
One shows that $G_{0}$ is finitely generated as in \cite[p. 7]{DoSa2}, with the help of lemma \ref{lemmeBirkhoff}. 
To show that it is free notice that a section of the kernel of the surjective map
$$(\cit [x,x^{-1},\theta ])^{\mu}\rightarrow G_{0}\rightarrow 0$$
is given by $\mu$ Laurent polynomials which vanish everywhere because, for every $x\in \mathcal{M}_{B}$, the sections defined in lemma \ref{lemmeBirkhoff} yield the basis of the Brieskorn lattice of $F_{x}:=F(. \; ,x)$ given by \cite[proposition 3.2]{DoSa2}. This gives the first assertion. 
Let us show the second one: the assertion about $\nabla^{B}_{\partial_{\theta}}$ is clear, thanks to the definition of the $\omega_{k}$'s. The action of $\nabla^{B}_{\partial_{x}}$ is defined, for $\eta\in G_{0}$, by 
 $$\nabla^{B}_{\partial_{x}}(\eta)=-u_{0}\eta\theta^{-1} +\mathcal{L}_{\partial_{x}}(\eta )$$ 
and we have, for $\eta =u_{0}u_{1}^{r_{1}}\cdots u_{n}^{r_{n}}\omega_{0}$,
$$u_{0}\eta =\frac{1}{\mu x}F\eta -\frac{1}{\mu x}\theta (\sum_{i=1}^{n}r_{i}-w_{i})\eta.$$
 We deduce from this, because $\theta^{2}\nabla^{B}_{\partial_{\theta}}$ is induced by the multiplication by $F$, that 
$$\nabla^{B}_{\partial_{x}}\omega_{k}=-\frac{A_{0}(x)}{\mu x}\theta^{-1} (\omega_{k})+\frac{1}{\mu x}( \mu +\sum_{i=1}^{n}a(k)_{i}-\sum_{i=1}^{n}w_{i}-\alpha_{k})\omega_{k}.$$
Now, one has $\sum_{i=1}^{n}a(k)_{i}=k-1$ (see section \ref{combinatoire}) and $\sum_{i=1}^{n}w_{i}=\mu -1$ so that
$$\mu +\sum_{i=1}^{n}a(k)_{i}-\sum_{i=1}^{n}w_{i}-\alpha_{k}=k-\alpha_{k}.$$
\end{proof}

\begin{remark}\label{start}(1) Put $x=1$. Lemma \ref{lemmeBirkhoff} yields the canonical (in the sense of \cite[Section 5]{DoSa2}) solution 
$\omega^{o}=(\omega_{0}^{o},\cdots ,\omega_{\mu -1}^{o})$
of the Birkhoff problem for the Brieskorn lattice of $f$ given by  
\cite[Proposition 3.2 and Proposition 5.2]{DoSa2}: the logarithmic lattice $E:=\cit [\tau ]<\omega_{0}^{o},\cdots ,\omega_{\mu -1}^{o}>$ is in one -to-one correspondence with M. Saito's canonical opposite filtration to the Hodge filtration on the space of vanishing cycles.\\
(2) The deformation $F$ can be seen as a 'rescaling' of the function $f$ and
it is possible to present the proof of the previous proposition in a slightly different way. However, we prefer to keep our more direct approach 
because it emphasizes the multiplication by $u_{0}$ (see the last part of section \ref{setting}) and gives the general way to proceed if one wants to compute other examples, e.g  $F(u_{1},u_{2},x)=u_{1}+u_{2}+\frac{1}{u_{1}u_{2}^{2}}+\frac{x}{u_{2}}.$\\ 
(3) In order to make the link with the $J$-function and quantum differential operators, notice that
$$[w^{w}\theta^{\mu}\prod_{i=1}^{\mu}(x\nabla_{\partial_{x}}-c_{i})-x]\omega_{0}=0$$ 
(compare with \cite[corollary 1.8]{Coa}).
$\blacklozenge$
\end{remark}

\begin{remark} (Various generalizations)\\
(1) The case $w_{0}\neq 1$ can be handled using the presentation of the Gauss-Manin system considered in \cite{dGMS}. This is longer but yields the same result: one has to 
replace $w_{1}^{a(k)_{1}}\cdots w_{n}^{a(k)_{n}}$ by $w_{0}^{a(k)_{0}}w_{1}^{a(k)_{1}}\cdots w_{n}^{a(k)_{n}}$ in the definition of the $\omega_{k}$'s and 
$w_{1}^{w_{1}}\cdots w_{n}^{w_{n}}$ by $w_{0}^{w_{0}}w_{1}^{w_{1}}\cdots w_{n}^{w_{n}}$ in the definition of $A_{0}(x)$.\\
(2) One could start more generally with the function
$$f(u_{1},\cdots ,u_{n})=b_{1}u_{1}+\cdots +b_{n}u_{n}+\frac{1}{u_{1}^{w_{1}}\cdots u_{n}^{w_{n}}}$$ 
where $b_{1},\cdots ,b_{n}$ are complex numbers such that $b_{1}\cdots b_{n}\neq 0$ and would obtain analoguous results. 
The Laurent polynomial considered in \cite{DoSa2} is obtained putting $b_{i}=w_{i}$ for all $i$ in $f$. But, if we keep in mind mirror symmetry, only the case $b_{i}=1$ will be really relevant. $\blacklozenge$
\end{remark}

The basis $\omega$ has another remarkable property: it yields a canonical extension of $G$ to $\cit^{*}\times\cit$.  
To see this, put $R:=\mu^{-1}(H-A_{\infty})$. It follows from section \ref{combinatoire} that
$$R=\Diag (c_{0} ,\cdots ,c_{\mu -1})$$
and from theorem \ref{basevarphi} that
the matrix of $x\nabla^{B}_{\partial_{x}}$ in the basis $\omega$ is given by
$$-\mu^{-1}\frac{A_{0}(x)}{\theta} +R.$$
Let ${\cal L}$ be the  $\cit [x, \theta ,\theta^{-1}]$-submodule of $G$ generated by $\omega$:
 $x\nabla^{B}_{\partial_{x}}$ induces a map on ${\cal
  L}/x{\cal L}$ whose eigenvalues are contained in $[0,1[$,
because $A_{0}(0)$ is a Jordan matrix and because $c_{k}\in [0,1[$ for
$k=0,\cdots ,\mu -1$. Thus we get

\begin{corollary}\label{coro:Deligne,extension}
The lattice ${\cal L}$ is Deligne's canonical extension of the Gauss-Manin system $G$ to $\cit^{*}\times \cit$  such that the eigenvalues of the residue of $\nabla^{B}_{\partial_{x}}$ are contained in $[0,1[$.\qed
\end{corollary}

 Theorem \ref{basevarphi} says that the basis $\omega$ gives an extension of $G_{0}$ as a trivial bundle $H^{B}$ on $\ppit^{1}\times \mathcal{M}_{B}$ (the module of its global sections is generated by $\omega_{0},\cdots ,\omega_{\mu -1}$) equipped with a connection $\nabla^{B}$ with logarithmic pole at $\tau :=\theta^{-1}=0$ and pole of order less or equal to two at $\theta =0$ (see for instance \cite[section 2.1]{Sab2}).
These are the first ingredients of our quantum differential system.

\subsubsection{Flat and orbifold bases}
\label{flatbasis}

Let $\Delta$ be an open disc in $\cit^{*}$ and, for $x\in\Delta$,
$\omega^{\plat}:=\omega x^{-R}$. $\omega^{\plat}$ is a local
basis of $G_{0}^{an}:={\cal O}_{\Delta}\otimes G$ and we will call it {\em a flat basis}, flat with
respect to the restriction $\bigtriangledown$ of $\nabla^{B}$ at $\{\theta
=\infty\}\times\cit^{*}$. The connection $\nabla^{B}$ in the basis
$\omega^{\plat}$ is
$$-\frac{A_{0}^{\plat}(x)}{\theta}\frac{dx}{\mu x}+\left(\frac{A_{0}^{\plat}(x)}{\theta} +A_{\infty}\right)\frac{d\theta}{\theta}$$
where
$$A_{0}^{\plat}(x)=\mu \left ( \begin{array}{cccccc}
0   & 0   & 0 & \cdots & 0   &  x^{1-c_{\mu -1}}/w^{w}\\
x^{c_{1}-c_{0}}   & 0   & 0 & \cdots & 0   & 0\\
0   &  x^{c_{2}-c_{1}} & 0 & \cdots & 0   & 0\\
..  & ... & . & \cdots & .   & .\\
..  & ... & . & \cdots & .   & .\\
0   & 0   & . & \cdots &  x^{c_{\mu -1}-c_{\mu -2}}  & 0 
\end{array} \right ),$$
the $c_{i}$'s being defined in section \ref{combinatoire}.

For $i\in\{0, \ldots ,\mu-1\}$, we denote
\begin{equation}\label{eq:chg,base,orb,flat,varphi}
  \omega^{\orb}_{i}:=s_{i}^{-1}\omega_{i}^{\plat}=x^{-c_{i}}s_{i}^{-1}\omega_{i}
\end{equation}
where the $s_{i}$ are defined in \eqref{eq:17}.
The  connection $\nabla^{B}$ in the basis $\omega^{\orb}$ is
$$-\frac{A_{0}^{\orb}(x)}{\theta}\frac{dx}{\mu x}+\left(\frac{A_{0}^{\orb}(x)}{\theta} +A_{\infty}\right)\frac{d\theta}{\theta}$$
where
$$A_{0}^{\orb}(x)=\mu \left ( \begin{array}{cccccc}
0   & 0   & 0 & \cdots & 0   &  a_{\mu}x^{1-c_{\mu -1}}\\
a_{1}x^{c_{1}-c_{0}}   & 0   & 0 & \cdots & 0   & 0\\
0   & a_{2}x^{c_{2}-c_{1}} & 0 & \cdots & 0   & 0\\
..  & ... & . & \cdots & .   & .\\
..  & ... & . & \cdots & .   & .\\
0   & 0   & . & \cdots &  a_{\mu -1}x^{c_{\mu -1}-c_{\mu -2}}  & 0 
\end{array} \right ),$$
the $a_{i}$'s being defined in \eqref{eq:20}.

\subsubsection{The pairing }
\label{dualite}

We define in this section a nondegenerate, symmetric and $\nabla^{B}$-flat bilinear form on $G_{0}$. 
The lattice $G_{0}^{o}$ is equipped with a nondegenerate bilinear form          
$$S^{o}:G_{0}^{o}\times G_{0}^{o}\rightarrow\cit[\theta ]\theta^{n},$$
$\nabla^{B,o}$-flat and satisfying , for $p(\theta )\in\cit[\theta ]$,
$$p(\theta)S^{o}(\cdot \, ,\, \cdot )=S^{o}(p(\theta)\cdot \, ,\, \cdot )=S^{o}(\cdot \, ,p(-\theta)\, \cdot ).$$ 
More precisely, in the basis $\omega^{o}=(\omega^{o}_{0},\cdots
,\omega^{o}_{\mu -1})$ of $G_{0}^{o}$ considered in remark \ref{start}
(1), one has
$$S^{o}(\omega^{o}_{k},\omega^{o}_{\ell})=\left\{ \begin{array}{ll}
S^{o}(\omega^{o}_{0},\omega^{o}_{n})\in\cit^{*} \theta^{n} &  \mbox{if $0\leq k\leq n$ and $k+\ell =n$,}\\
{w^{-w}}S^{o}(\omega^{o}_{0},\omega^{o}_{n}) & \mbox{if  $n+1\leq k\leq \mu -1$ and $k+\ell =\mu +n$,}\\
0 & \mbox{otherwise}
\end{array}
\right .$$
\noindent where $w^{w}=w_{1}^{w_{1}}\cdots w_{n}^{w_{n}}$ as above.
This is shown as in \cite[Sect. 4]{DoSa2}. From now on, we will choose
the normalization
$S^{o}(\omega^{o}_{0},\omega^{o}_{n})=1/m_{1}\theta^{n}$ (recall that
$m_{1}=w_{1}\cdots w_{n}$).

We define, in the basis $\omega$ given by
theorem \ref{basevarphi},
\begin{align}\label{eq:paring,omega,phir}
  S^{B}(\omega_{k},\omega_{\ell})=\left\{ 
\begin{array}{ll}
{\theta^{n}m_{1}^{-1}} &  \mbox{if $0\leq k\leq n$ and $k+\ell =n$,} \\
{\theta^{n}m_{1}^{-1}xw^{-w}} & \mbox{if  $n+1\leq k\leq \mu -1$ and $k+\ell =\mu +n$,}\\
0 & \mbox{otherwise}
\end{array}
\right .
\end{align}
This gives                                                                       
$$S^{B}:G_{0}\times G_{0}\rightarrow\cit[x,x^{-1},\theta ]\theta^{n}$$
by linearity, using the rules        
$$a(x, \theta )S(\cdot \, ,\, \cdot )=S(a(x,\theta )\cdot \, ,\, \cdot )=S(\cdot \, ,a(x,-\theta )\, \cdot )$$ 
for $a(x,\theta )\in\cit [x, \theta ]$. Flatness is defined by equations (\ref{eq:13}), (\ref{eq:14}) (replacing $z$ by $\theta$ and $\partial_{k}$ by $\partial_{x}$).
The following lemma justifies the definition of $S^{B}$:

\begin{lemma}\label{symetrie} 
 The bilinear form $S^{B}$ is $\nabla^{B}$-flat.
\end{lemma}
\begin{proof} We work in the basis $\omega$: it follows first from the definition of $A_{0}(x)$ and $S^{B}$ that 
one has $(A_{0}(x))^{*}=A_{0}(x)$  where $^{*}$ denotes the adjoint with respect to $S^{B}$. The symmetry property of the numbers $\alpha_{k}$ (see corollary \ref{lesalphak}) shows also that $A_{\infty}+A_{\infty}^{*}=nI$. This gives equation (\ref{eq:13}). Now, equation (\ref{eq:14}) reads 
$$x\partial_{x}S^{B}(\omega_{i}, \omega_{j})=S^{B}(R(\omega_{i}), \omega_{j})+S^{B}(\omega_{i}, R(\omega_{j}) )$$
but this follows once again from lemma \ref{lesalphak}. 
\end{proof}

\begin{corollary}\label{orbipairing}
We have 
$$S^{B}(\omega^{\orb}_{k},\omega^{\orb}_{\ell})=
\left\{ \begin{array}{ll}
{m_{1}^{-1}}\theta^{n} & \mbox{if $0\leq k\leq n$ and $k+\ell =n$,}\\
{m_{i+1}^{-1}}\theta^{n} & \mbox{if $d_{1}+\cdots +d_{i}\leq k <d_{1}+\cdots +d_{i+1}$ and $k+\ell =\mu+n$,}\\
0 & \mbox{otherwise}
\end{array}
\right .$$
\end{corollary}
\begin{proof} By lemma \ref{symetrie}, $S^{B}$ is constant in the basis $\omega^{\plat}$ thus in the basis $\omega^{\orb}$ and the result follows from the definitions, using the fact that $m_{i}=m_{j}$ if $i+j=k+2$ and $m_{1}\cdots m_{k}=w^{w}$. 
\end{proof}

\begin{remark} \label{gsurJacobi}
(1) The coefficient of $\theta^{n}$ in $S^{B}(\varepsilon ,\eta )$, $\varepsilon ,\eta \in G_{0}$, depends only on the classes of $\varepsilon$ and  $\eta$ in $G_{0}/\theta G_{0}$. We will denote it by $g([\varepsilon ],[\eta ])$. This defines a nondegenerate bilinear form $g$ on $G_{0}/\theta G_{0}$, see \cite[p. 211]{Sab1}.\\
(2) The bilinear form $S^{B}$ defines a bilinear form (also denoted by $S^{B}$) on the trivial bundle $H^{B}$ (see for instance \cite[section 1.4]{Sab2}).  $\blacklozenge$
\end{remark}

\subsection{R\'esum\'e (the $B$-model quantum differential system)}
\label{Saitostructure}

We have constructed a  trivial bundle $H^{B}$ (section \ref{canonicalbasis}), equipped with a flat meromorphic connection $\nabla^{B}$, and a $\nabla^{B}$-flat pairing $S^{B}$ (section \ref{dualite}). Summarizing, we get

\begin{theorem}\label{BcanonicalMSstructure}   The tuple 
$$\mathcal{S}^{B}_{w}=\left(\mathcal{M}_{B}, H^{B}, \nabla^{B}, S^{B},
  n\right)$$
is a quantum differential system.
\end{theorem}

\begin{definition} We will say that $\mathcal{S}^{B}_{w}$ is the {\em
  small $B$-model quantum differential system}. 
\end{definition}

\begin{remark}\label{BmodIrvsDM}
Iritani's $B$-model $D$-module (see \cite[definition 3.16]{Ir}) is different (compare with remark \ref{AmodIrvsDM}), as he deals only with bundles on $\cit\times\mathcal{M}_{B}$: in particular, he doesn't consider the Birkhoff problem at all.$\blacklozenge$
\end{remark}

\section{The mirror partner of the small quantum orbifold cohomology
  of $\ppit (w)$}
\label{mirrorsmall}

\subsection{Correspondence}\label{sec:correspondance}

Let us first summarize the results obtained. 
On both sides we have a trivial bundle over a base isomorphic to $\ppit^{1}\times \cit^{*}$. 
The free $\cit [q,q^{-1}]$-module $H_{A}$ of global sections of
$\widetilde{H}^{A,\smal}$ is generated by $(P^{\bullet})^{j}$ for $j=0,\cdots ,\mu -1$ whereas the free $\cit [x,x^{-1}]$-module $H_{B}$ of 
global sections of $H^{B}$ is generated by $(\omega_{i})$.
The following theorem gives an explicit isomorphism between the small $A$-model quantum differential system and the 
small $B$-model quantum differential system and by the way a precise form of the mirror theorem for weighted projective spaces.

\begin{theorem} 
\label{quantum}
The map 
$$\gamma : H_{A} \rightarrow H_{B}$$
defined by 
\begin{align}\label{eq:19}
  \gamma (P^{\bullet j})=\omega_{j}.
\end{align}
gives an isomorphism between $H_{A}$ and $H_{B}$, after identifying $\ppit^{1}\times\mathcal{M}_{A}$ and $\ppit^{1}\times\mathcal{M}_{B}$ via 
the map $(z,q)\mapsto (\theta ,x)$. It yields an isomorphism between the small $A$-model quantum differential system 
$$(\mathcal{M}_{A}, \widetilde{H}^{A,\smal}, \widetilde{\nabla}^{A,\smal}, \widetilde{S}^{A,\smal},n)$$ 
and the small $B$-model quantum differential system  $$(\mathcal{M}_{B}, H^{B}, \nabla^{B}, S^{B},n).$$ 
\end{theorem}

\begin{remark} \label{rem:correspondence} Identify
  $\mathcal{M}=\mathcal{M}_{A}=\mathcal{M}_{B}$.  Proposition 4.8 of
  Iritani \cite{Ir} implies that our two $D$-modules are isomorphic
  over $\cit\times\mathcal{M}$.  So our result above is about the
  compatibility of the extensions over $\ppit^{1}\times\mathcal{M}$.
  Namely, the natural extension on the A-side (recall that the small
  $A$-model D-module is naturally defined over $\ppit^{1}\times
  \mathcal{M}$) corresponds to the solution of the Birkhoff problem
  given in Theorem \ref{basevarphi}.  More precisely, the isomorphism
  over $\cit \times\mathcal{M}$ of Proposition 4.8 of Iritani \cite{Ir}
  for the A-side (for the B-side, one has to take $\nabla^{B}$ and
  replace the unit $\phi_{0}$ by $\omega_{0}$) is the following :
  \begin{align*}
    \cit[q^{\pm},z]\langle zq\partial_{q}\rangle/ \langle T_{w}
    \rangle
    &\stackrel{\sim}{\longrightarrow}  \left(\widetilde{H}^{A,\smal}, \widetilde{\nabla}^{A,\smal}\right)\\
    P(q,z,zq\partial_{q})&\longmapsto
    P(q,z,\widetilde{\nabla}^{A,\smal}_{zq\partial_{q}}) \phi_{0}
\end{align*}
where $T_{w}=\prod_{i=1}^{\mu}(zq\partial_{q}-zc_{i}) -qw^{-w}$ (see
Corollary 1.8 in \cite{Coa}).  The natural choice of basis in this framework is thus
$((\widetilde{\nabla}^{A,\smal}_{zq\partial_{q}})^{i}\phi_{0})_{i=0, \ldots ,\mu-1}$. It gives an extension on
$\ppit^{1}\times \mathcal{M}$ but this extension
will not give a quantum differential system because the connection does not have
{\em a priori} a logarithmic pole along $\{z=\infty \}\times M$ (see formula
\eqref{eq:log+rk1,connection}). Indeed the matrix of the connection is
the companion matrix associated to
\begin{displaymath}
  T_{w}=(zq\partial_{q})^{\mu}+\sum_{i=1}^{\mu}(-z)^{i}\sigma_{i}(c_{1}, \ldots ,c_{\mu})(zq\partial_{q})^{\mu-i}-qw^{-w}
\end{displaymath}
where $\sigma_{i}$ are elementary symmetric polynomials.  As
$c_{1}=\cdots=c_{n}=0$, we have
$\sigma_{i}(c_{1}, \ldots ,c_{\mu})=0$ for $i$ in $\{\mu-n+1, \ldots ,\mu\}$, so, in the basis
$(\widetilde{\nabla}^{A,\smal}_{zq\partial_{q}})^{i}\phi_{0}$, we have
:
\begin{align*}
  \widetilde{\nabla}^{A,\smal}_{q\partial_{q}}=q\partial_{q}+\frac{1}{z}B_{-1}(q)+B_{0}(q)+\cdots+z^{\mu-n-1}B_{\mu-n-1}(q)
\end{align*}
which is not of the form \eqref{eq:log+rk1,connection}.  In the case $\ppit (1,2,2)$ we can verify for instance that $B_{1}(q)$ is not the zero matrix 
(it has a coefficient $\frac{1}{4}$ on the last column).
Notice that
Guest and Sakai consider an analogous problem in \cite{Guest} and they solve
it using the so-called ``Birkhoff factorization'' (see \cite{Guest-topology} or chapter
6 in \cite{Guest-book}). In general finding the good extension is a
difficult problem but in our case, it can be done.   $\blacklozenge$
\end{remark}

\begin{proof}[Proof of Theorem \ref{quantum}] 
From Proposition \ref{prop:nabla,basis,nonflat} and Theorem \ref{basevarphi},
the matrices of the connections in the bases $(P^{\bullet j})$ and $(\omega_{i})$ are the same.
 For the pairing, it is enough to notice that 
$$\widetilde{S}^{A,\smal}(P^{\bullet i}, P^{\bullet j}) =S^{B}(\gamma (P^{\bullet i}), \gamma (P^{\bullet j}))$$
but this follows from the formula \eqref{eq:paring,omega,phir} and proposition \ref{prop:pairingglobal}.  
\end{proof}

\begin{remark} The definition of $\gamma$ in \eqref{eq:19} identifies
  $P^{\bullet j}\leftrightarrow \omega_{j}$ for $j\in\{0, \ldots
  ,\mu-1\}$. This also implies that the flat sections
  $\omega_{i}^{\orb}$ (see \eqref{eq:chg,base,orb,flat,varphi}) are
  identified with the flat sections $\mathbf{1}_{c_{i}}P^{r(i)}$ (see
  Remark \ref{rem:A,side,flat}) where $r(i):=\#\{k\mid k<i \mbox{ and }
  c_{k}=c_{i}\}$.
$\blacklozenge$
\end{remark}

\noindent We can thus identify the $A$-model quantum differential system $\mathcal{S}^{A}_{w}$ and the $B$-model quantum differential system $\mathcal{S}^{B}_{w}$: 
the result is a quantum differential system which will we denote by 
$$\mathcal{S}_{w}:=(\mathcal{M},H,\nabla ,S, n).$$
We also get, with the help of proposition \ref{MStoSTF}, a  Frobenius type structure 
$$\fit_{w}=(\mathcal{M}, E, \bigtriangledown , R_{0}, R_{\infty},\Phi ,g)$$
on $\mathcal{M}$ where $E:= G_{0}/\theta G_{0}=\Omega^{n}(U)[x,x^{-1}]/d_{u}F\wedge\Omega^{n-1}(U)[x,x^{-1}]$.

\begin{definition} 
\label{MSS}
(1) The tuple $\mathcal{S}_{w}$ is called the $w$-quantum differential system.\\
(2) The tuple $\fit_{w}$ is called the $w$-Frobenius type structure. 
\end{definition}

\noindent We will use these objects in order to get Frobenius manifolds.

\subsection{The small quantum product and the Jacobian ring} 

Using thoerem \ref{quantum} we can give an interpretation of the small quantum product in terms of a product on a Jacobian ring, that is 
in terms of commutative algebra.\\

For $k=0,\cdots ,\mu -1$, put
$\omega_{k}=g_{k}\omega_{0}$ where
$g_{0}=1$  and
$$g_{k}=\frac{x}{w^{a(k)}}u^{a(k)}$$
for $k=1,\cdots ,\mu -1$ (see section \ref{deformation}). We define now the product $*$ on $E:=G_{0}/\theta G_{0}$
by
\begin{align}\label{eq:22}
[\omega_{i}]*_{x}[\omega_{j}]:=[g_{i}g_{j}\omega_{0}]
\end{align}
where $[\ ]$ denotes the class in $E$, which we identify, using $\omega_{0}$, to the Jacobian ring 
\begin{displaymath}
 \frac{\cit[x,x^{-1}][u_{1},u_{1}^{-1},\cdots ,u_{n},u_{n}^{-1}]}{(\frac{\partial F}{\partial u_{1}},\cdots , \frac{\partial F}{\partial u_{n}})}.
\end{displaymath}

\begin{proposition}\label{prod}

  Let $i,j\in \{0,\cdots ,\mu -1\}$. If $i+j\geq \mu$, we denote
  $\overline{i+j}:=i+j-\mu$.\\
  (1) We have, in $E$,
\begin{align}\label{prodvarphi}
[\omega_{i}]*_{x}[\omega_{j}]=\left\{ \begin{array}{ll}
[\omega_{i+j}] &  \mbox{if $i+j\leq \mu -1$,} \\
\frac{x}{w^{w}}[\omega_{\overline{i+j}}] & \mbox{if  $i+j\geq \mu$}
\end{array}
\right .
\end{align}
In particular, $[\omega_{i}]=[\omega_{1}]^{*i}:= \underbrace{[\omega_{1}]*_{x}\cdots *_{x}[\omega_{1}]}_{\mbox{$i$ times}}$.\\
(2) We have, in $H_{\orb}^{*}(\ppit(w),\cit )$, 
\begin{align}\label{prodPi}
P^{\bullet i}\bullet_{q} P^{\bullet j}=\left\{ \begin{array}{ll}
P^{\bullet (i+j)} &  \mbox{if $i+j\leq \mu -1$,} \\
\frac{q}{w^{w}}P^{\bullet (\overline{i+j})} & \mbox{if  $i+j\geq \mu$}
\end{array}
\right .
\end{align}
\end{proposition}
\begin{proof}
  (1) Because $u_{0}u_{1}^{w_{1}}\cdots u_{n}^{w_{n}}=1$ and, for
  $i\geq 1$,
  $\frac{u^{a(i)}}{w^{a(i)}}\omega_{0}=x^{i-1}u_{0}^{i}\omega_{0}$
  in $E$. (2) Follows from proposition \ref{prop:nabla,basis,nonflat}.
\end{proof}

\noindent Notice that the matrix $\frac{1}{\mu}A_{0}(x)$ in theorem \ref{basevarphi} represents the endomorphism $[\omega_{1}]*_{x}$ in the basis $[\omega]$.

\noindent At the end, we get the announced relationship:

\begin{corollary}\label{productJacobi} The product $*_{x}$ is the mirror partner of the small quantum product $\bullet_{q}$: we have 
$$[\gamma (P^{\bullet i})]*_{x}[\gamma (P^{\bullet j})]=[\gamma (P^{\bullet i}\bullet_{q} P^{\bullet j})].$$
\end{corollary}
\begin{proof} Follows from proposition \ref{prod} and the definition of $\gamma$.
 \end{proof}

\section{Limits}

\label{limite}

Up to now, we have worked on $\mathcal{M}=\cit^{*}$ and we want now to define a
limit at $0$ of the structure $\mathcal{S}_{w}$ ({\em resp.} $\fit_{w}$). This
should be of course a quantum differential system ({\em resp.} a Frobenius type structure)
on $\ppit^{1}$ ({\em resp.} on a point), as canonical as possible.
This limit
will be constructed with the help of the Kashiwara-Malgrange
$V$-filtration at the origin. 
The desired limit Frobenius type structure (on a
point) will be then obtained using proposition \ref{MStoSTF}.  

Usually on the A-side, one recovers the cup product from the quantum
product setting $q=0$. This works nicely in the basis
$(\mathbf{1}_{f_{i}}P^{j})$. Nevertheless, when one works with the quantum
differential system, like we do, the natural basis is not
$(\mathbf{1}_{f_{i}}P^{j})$ but $(P^{\bullet j})$ (see Remark
\ref{rem:A,side,flat}). As $(P^{\bullet j})$ depends on $q$, it make
no sense to set directly $q=0$. For example the matrix $C(q=0)$ is not
the endomorphism $P\cup_{\orb}$. So to recover the limit at ``q=0'',
we need to give a grading by the Kashiwara-Malgrange $V$-filtration (see {\em f.i} \cite[2.e and A.b.3]{DoSa1} for the definition of this filtration).

\subsection{Canonical limits of the structures $\mathcal{S}_{w}$ and $\fit_{w}$}

\label{limSTF}

We apply the recipe announced above. For convenience reasons, we start from the
$B$-model and we use the notations of section \ref{Bmodel}, forgetting the index $B$.

\subsubsection{The $V$-filtration at $x=0$}
Recall the basis $\omega =(\omega_{0},\cdots ,\omega_{\mu -1})$ of $G_{0}$ over $\cit [x, x^{-1},\theta ]$, which is also a basis of $G$ over $\cit [x, x^{-1},\theta ,\theta^{-1}]$.
Put $v(\omega_{0})=\cdots =v(\omega_{n})=0$ and, for $k=n+1,\cdots ,\mu -1$,
$v(\omega_{k})=c_{k}$. Define, for $0\leq \alpha <1$,
$$V^{\alpha}G=\sum_{\alpha\leq v(\omega_{k})}\cit [x][\theta ,\theta^{-1}]\omega_{k}+
x\sum_{\alpha >v({\omega}_{k})}\cit [x][\theta ,\theta^{-1}]\omega_{k},$$
$$V^{>\alpha}G=\sum_{\alpha < v(\omega_{k})}\cit [x][\theta ,\theta^{-1}]\omega_{k}+
x\sum_{\alpha \geq v(\omega_{k})}\cit [x][\theta ,\theta^{-1}]\omega_{k}$$
and $V^{\alpha +p}G=x^{p}V^{\alpha}G$ for $p\in\zit$ and $\alpha\in [0,1[$. This gives a decreasing filtration $V^{\bullet}$ of $G$ 
by $\cit [x][\theta ,\theta^{-1}]$-submodules
such that
$$V^{\alpha}G=\cit [\theta ,\theta^{-1}]\langle\omega_{k}|v(\omega_{k})=\alpha\rangle+V^{>\alpha}G.$$
Notice that the lattice ${\cal L}$ (see section \ref{canonicalbasis}) is equal to $V^{0}G$  and that ${\cal L}/x{\cal L}=V^{0}G/V^{1}G$. We will put $G^{\alpha}:=V^{\alpha}G/V^{>\alpha}G$ and $\overline{G}:=\oplus_{\alpha\in [0,1[}G^{\alpha}$.

\begin{lemma} \label{Jordan}
(1) For each $\alpha$, $(x\nabla_{\partial_{x}}-\alpha )$ is nilpotent on $G^{\alpha}$.\\
(2) Let $N$ be the nilpotent endomorphism of $\overline{G}$ which restricts to $(x\nabla_{\partial_{x}}-\alpha )$ on $G^{\alpha}$. Its Jordan blocks are in one to one correspondence with the maximal constant sequences in $(c_{0},\cdots ,c_{\mu -1})$ and the corresponding sizes are the same.\\
(3) The classes $[\omega_{0}],\cdots ,[\omega_{\mu -1}]$ give a basis $[\omega]$ of $\overline{G}$ over $\cit [\theta ,\theta^{-1}]$.
\end{lemma}
\begin{proof} (1) It suffices to prove the assertion for $\alpha \in [0,1[$. It follows from theorem \ref{basevarphi} that we have
$$x\nabla_{\partial_{x}}\omega_{k}=-\frac{1}{\theta}\omega_{k+1}$$
for $k=0,\cdots ,n-1$ and $x\nabla_{\partial_{x}}\omega_{n}\in V^{>0}G$. Moreover we have, for $k=n+1,\cdots ,\mu -2$,  
$$(x\nabla_{\partial_{x}}-c_{k})\omega_{k}=-\frac{1}{\theta}\omega_{k+1}$$
and this is equal to $0$ in $G^{v(\omega_{k})}$ if  $c_{k+1}> c_{k}$. Last, 
$$(x\nabla_{\partial_{x}}-c_{\mu -1})\omega_{\mu -1}
=-\frac{1}{\theta}{x}{w^{-w}}\omega_{0}
\in x\sum_{v(\omega_{\mu -1})\geq v(\omega_{k})}\cit [x]\omega_{k}\subset V^{>c_{\mu -1}}G.$$
(2) follows from (1) and (3) follows from the definition of $V^{\bullet}$.
\end{proof}

\noindent The matrix of $N$ in the basis $[\omega]$ is ${B}{\theta^{-1}}$ where
$B_{i,j}=0$ if $i\neq j+1$, $B_{i+1, i}=-1$ if $c_{i}=c_{i-1}$ and $B_{i+1, i}=0$ if $c_{i}\neq c_{i-1}$ (notice that $-\mu B=A_{0}^{\plat}(0)$).\\

\begin{corollary}
The filtration $V^{\bullet}$ is the Kashiwara-Malgrange filtration at $x=0$.  
\end{corollary}

\begin{proof}
By the previous lemma, the filtration $V^{\bullet}$ satisfies all the characteristic properties of the Kashiwara-Malgrange filtration. 
\end{proof}

\subsubsection{Limits}
\label{limitSTF}

The free $\cit [\theta ,\theta^{-1}]$-module $\overline{G}$ is equipped with a connection $\overline{\nabla}$ in the basis $[\omega]$ is
$$\left(\frac{\overline{A}_{0}}{\theta}+A_{\infty}\right)\frac{d\theta}{\theta}$$
where $\overline{A}_{0}=-\mu B$ and $A_{\infty}=\Diag
(\alpha_{0},\cdots ,\alpha_{\mu -1})$.  
We now need a limit bilinear form.
Let $\overline{G}_{0}$ be the $\cit
[\theta]$-submodule of $\overline{G}$ generated by
$[\omega_{0}],\cdots ,[\omega_{\mu-1}]$ and define
$$\overline{S}:\overline{G}_{0}\times \overline{G}_{0}\rightarrow \cit [\theta ]\theta^{n}$$
by
$$\overline{S}([\omega_{k}], [\omega_{n-k}])=\frac{1}{w_{1}\cdots w_{n}}\theta^{n}$$
for $k=0,\cdots ,n$ (in which case $c_{k}=c_{n-k}=0$),
$$\overline{S}([\omega_{k}], [\omega_{\mu +n-k}])=\frac{1}{w_{1}^{w_{1}+1}\cdots w_{n}^{w_{n}+1}}\theta^{n}$$
for $k=n+1,\cdots ,\mu -1$ (in which case $c_{k}+c_{\mu +n-k}=1$) and $\overline{S}([\omega_{i}],[\omega_{j}])=0$ otherwise. 
The pairing $\overline{S}$ is induced by $S$ on $\overline{G}$ (hence it is indeed a limit): this is shown as in \cite[remark 3.6]{Sab3} (with only mild modifications) because
$$S(V^{\beta}G, V^{1-\beta}G)\subset x\cit [x,\theta ,\theta^{-1}]$$
if $\beta\neq 0$ (and thus the induced bilinear form on the graded pieces is obtained taking the coefficient of $x$) and
$$S(V^{0}G, V^{0}G)\subset \cit [x,\theta ,\theta^{-1}]$$
where $V^{\bullet}$ is the Kashiwara-Malgrange filtration at $x=0$ defined above.

As in section \ref{deformation}, we get an extension of $\overline{G}_{0}$ as a trivial bundle $\overline{H}$ on $\ppit^{1}$, equipped with a connection $\overline{\nabla}$ and a pairing $\overline{S}$.

\begin{theorem}
The tuple
$\overline{\mathcal{S}}_{w}=(\overline{H}, \overline{\nabla} , \overline{S}, n)$ 
is a quantum differential system on $\ppit^{1}$.
\end{theorem}

\begin{proof}
  It is remains to show that $\overline{S}$ is $\overline{\nabla}$-flat, and it
  is enough to show that $(\overline{A}_{0})^{*}=\overline{A}_{0}$ and
  ${A}_{\infty}+({A}_{\infty})^{*}=n \id$.  The second
  equality follows easily from lemma \ref{lessk} and from the definition of
  $\overline{S}$. To show the first one, use moreover lemma \ref{Jordan}, the
  key point being that
  $\overline{S}(\overline{A}_{0}([\omega_{n}]),[\omega_{j}])=0=\overline{S}([\omega_{n}],\overline{A}_{0}([\omega_{j}]))$
  because, by lemma \ref{Jordan},
  $\overline{A}_{0}([\omega_{n}])=0$ and because
  $[\omega_{0}]$ does not belong to the image of
  $\overline{A}_{0}$.
\end{proof}
\begin{remark} It should be emphasized that the conclusion of the previous theorem is not always true if we work directly on ${\cal L}/x{\cal L}$,
that is if we forget the $gr^{V}$. $\blacklozenge$ 
\end{remark}

\begin{definition} \label{canonicallimitST}The tuple $\overline{\mathcal{S}}_{w}$
is the limit quantum differential system.
\end{definition}

Define now $\overline{E}=\overline{G}_{0}/\theta \overline{G}_{0}$ and let
$[\![\omega]\!]$ be the basis of $\overline{E}$ induced by
$[\omega]$. As explained in section \ref{FrobeniusSaito},
$\overline{E}$ is thus equipped with two endomorphisms
$\overline{R}_{0}$ and $\overline{R}_{\infty}$ (with respective matrices $\overline{A}_{0}$ and $-A_{\infty}$) and with a
nondegenerate bilinear form $\overline{g}$ obtained from $\overline{S}$ as in
remark \ref{gsurJacobi}.

\begin{corollary} \label{existlimFTS} The tuple 
 $$\overline{\fit}_{w}=(\overline{E}, \overline{R}_{0},\overline{R}_{\infty},\overline{g})$$
is a Frobenius type structure on a point. 
\end{corollary}

\begin{definition} \label{canonicallimitSTF} $\overline{\fit}_{w}$
is the limit Frobenius type structure.
\end{definition}

\begin{remark}\label{limitepreprim}
Let $(E, A, B, g)$ be a Frobenius type structure on a point.
We will say that an element $e$ of  $E$  
is a {\em pre-primitive section} if $(e,A(e),\cdots , A^{\mu -1}(e))$ is a basis of $E$ 
over $\cit$ and that $e$ is {\em homogeneous} if it is an eigenvector of $B$. 
Recall that $[\![\omega_{0}]\!]$ denotes the class of $\omega_{0}$ in $E$.
Then $[\![\omega_{0}]\!]$ is a pre-primitive and homogeneous section of the limit Frobenius type structure  
$(E, \overline{R}_{0}, \overline{R}_{\infty}, \overline{g})$ if and only if $\mu =n+1$. If $\mu\geq n+2$, this Frobenius type structure has no pre-primitive section at all.  $\blacklozenge$
\end{remark}

\subsection{Application: the mirror partner of the orbifold cohomology ring}
\label{orbifoldcohomology}

Recall that we have defined a product $\ast_{x}$ on $E:=G_{0}/\theta G_{0}$ (see \eqref{eq:22}).
The filtration
$(V^{\alpha})_{\alpha\in \rit}$ induces a decreasing filtration on
$E$, denoted by $(V^{\alpha}E)_{ \alpha\in \rit}$, which is compatible
with the product $\ast_{x}$, {\em i.e.} $V^{\alpha}E\ast_{x}V^{\beta}E\subset V^{\alpha+\beta}E$. 
As for the filtration $V^{\alpha}G$, we have $V^{\alpha+p}E=x^{p}V^{\alpha}E$ for any
$\alpha\in \rit$ and any $p\in \nit$. The vector space
$\overline{E}:=\overline{G}_{0}/\theta \overline{G}_{0}$ defined above
is also $\oplus_{\alpha\in [0,1[}\gr^{V}_{\alpha}E$.  We define a
product, denoted by $\cup$, on $\overline{E}$ by first graduating the
product $\ast_{x}$ on $\oplus_{\alpha\in \rit}\gr^{V}_{ \alpha}E$ and
then shifting it on $\overline{E}:=\oplus_{\alpha\in [0,1[}\gr^{V}_{\alpha}E$ by multiplying by an appropriate $x^{p}$.
In this way, proposition \ref{prod} implies that
$$[\![\omega_{i}]\!]\cup [\![\omega_{j}]\!]:=\frac{1}{w^{w}}[\![\omega_{\overline{i+j}}]\!]\ \mbox{if $i+j\geq \mu$ and $1+c_{\overline{i+j}}=c_{i}+c_{j}$},$$
$$[\![\omega_{i}]\!]\cup [\![\omega_{j}]\!]:=[\![\omega_{i+j}]\!]\ \mbox{if $i+j\leq \mu -1$ and $c_{i+j}=c_{i}+c_{j}$}$$  
and $[\![\omega_{i}]\!]\cup [\![\omega_{j}]\!]=0$ otherwise.
This product is homogeneous and $[\![\omega_{0}]\!]$ is the unit.  The bilinear form $\overline{g}$ on $\overline{E}$ is also homogeneous
because
$\overline{g}([\![\omega_{i}]\!],[\![\omega_{j}]\!])\neq 0$
only if $i+j=n$ or if $i+j=\mu +n$: in any case, $\alpha_{i}+\alpha_{j}=n$.

\begin{proposition}\label{orbicohring}
The tuple $(\overline{E},\cup ,\overline{g})$ is a Frobenius algebra, isomorphic to 
$$(H_{\orb}^{*}(\ppit (w),\cit), \cup_{\orb}, \langle \ .\ ,\ .\ \rangle).$$ 
\end{proposition}
\begin{proof} To prove the first assertion, it remains to show the compatibility condition  
$$\overline{g}([\![\omega_{i}]\!]\cup [\![\omega_{k}]\!],[\![\omega_{j}]\!])=\overline{g}([\![\omega_{i}]\!],[\![\omega_{j}]\!]\cup [\![\omega_{k}]\!])$$
but this follows from a straightforward computation of the right term and the left term, keeping in mind the definition of $\overline{g}$ and $\cup$. The second follows from section \ref{mirrorsmall}: the isomorphism is induced by $\gamma$.
\end{proof}

\noindent Of course, this result should be compared with \cite[Theorem 1.1]{Man}.

\begin{example}\label{exampleprod} $w_{0}=1, w_{1}=w_{2}=2$: the table of the orbifold cup-product $\cup_{\orb}$ is 

\begin{center}

\begin{tabular}{|c ||c|c|c|c|c|} \hline
$\cup_{\orb}$ & $\mathbf{1}$ & $P$ & $P^{2}$ & $\mathbf{1}_{\frac{1}{2}}$ & $\mathbf{1}_{\frac{1}{2}}P$ \\ \hline\hline
$\mathbf{1}$     & $\mathbf{1}$ & $P$ & $P^{2}$ & $\mathbf{1}_{\frac{1}{2}}$ & $\mathbf{1}_{\frac{1}{2}}P$ \\ \hline 
$P$     &  & $P^{2}$ & $0$ & $\mathbf{1}_{\frac{1}{2}}P$ & $0$ \\ \hline 
$P^{2}$ &  &  & $0$ & $0$ & $0$ \\ \hline 
$\mathbf{1}_{\frac{1}{2}}$ &  &  &  & $P$ & $P^{2}$ \\ \hline 
$\mathbf{1}_{\frac{1}{2}}P$ &  &  &  &  & $0$ \\ \hline  
\end{tabular}
\end{center} 

\noindent and the one of $\cup$ is 

\begin{center}
\begin{tabular}{|c ||c|c|c|c|c|} \hline
$\cup$ & $[\![\omega_{0}]\!]$ & $[\![\omega_{1}]\!]$ & $[\![\omega_{2}]\!]$ & $[\![\omega_{3}]\!]$ & $[\![\omega_{4}]\!]$ \\ \hline\hline
$[\![\omega_{0}]\!]$     & $[\![\omega_{0}]\!]$ & $[\![\omega_{1}]\!]$ & $[\![\omega_{2}]\!]$ & $[\![\omega_{3}]\!]$ & $[\![\omega_{4}]\!]$ \\ \hline 
$[\![\omega_{1}]\!]$     &  & $[\![\omega_{2}]\!]$ & $0$ & $[\![\omega_{4}]\!]$ & $0$ \\ \hline 
$[\![\omega_{2}]\!]$ &  &  & $0$ & $0$ & $0$ \\ \hline 
$[\![\omega_{3}]\!]$ &  &  &  & $\frac{1}{16}[\![\omega_{1}]\!]$ & $\frac{1}{16}[\![\omega_{2}]\!]$ \\ \hline 
$[\![\omega_{4}]\!]$ &  &  &  &  & $0$ \\ \hline  
\end{tabular}
\end{center} 
 Recall that via mirror symmetry, $[\![\omega_{i}]\!]$ corresponds to
 $P^{\bullet i}$. So the difference of the constants between the two tables comes from the relation between $P^{\bullet i}$ and $(\mathbf{1}_{f_{i}}P^{j})$ (see
 Lemma \ref{lem:coates}).

Let us explain for instance the computation $[\![\omega_{3}]\!]\cup[\![\omega_{3}]\!]=[\![\omega_{1}]\!]/16$. By Proposition \ref{prod}, we have
\begin{align}
  \label{eq:23}
  [\omega_{3}]\ast_{x}[\omega_{3}]=\frac{x}{16}[\omega_{1}]
\end{align}
We have also $[\omega_{3}] \in V^{1/2}E$ and $x[\omega_{1}] \in
V^{1}E$ and the equality above is still true in the graded space
$\oplus_{\alpha\in\rit}\gr^{V}_{ \alpha}E$.  
As $\gr^{V}_{1}E:=x\gr^{V}_{0}E$, and because in $\overline{E}$ we only consider the graded pieces
between $[0,1[$, we
deduce that
$[\![\omega_{3}]\!]\cup[\![\omega_{3}]\!]=[\![\omega_{1}]\!]/16$.
Notice that putting $x=0$ in \eqref{eq:23}, we do not get the expected
result.  Doing the same computation on the $A$-side, we get
$P^{\bullet 3}\cup_{\orb}P^{\bullet3}=P/16$. Let us stress again that
setting directly $q=0$ does not give the right answer.
\end{example}

\section{Construction of Frobenius manifolds} 
\label{construction}

First, we recall how to construct Frobenius manifolds, starting from a
Frobenius type structure (our references will be \cite{Do1} and \cite{HeMa}): one needs a
homogeneous and primitive section yielding an invertible period map. 
We then use this construction to define a limit Frobenius manifold, by unfolding the limit Frobenius type structure $\overline{\fit}_{w}$ defined in section \ref{limSTF}. Last, we end with a discussion about {\em logarithmic} Frobenius manifolds, as defined in \cite{R}.

\subsection{Frobenius manifolds on $\mathcal{M}=\cit^{*}$}
\label{constructionfrobman}

Let $\Delta$ be an open disc in $\mathcal{M}$. The $w$-Frobenius type structure $\fit_{w}$ (see definition \ref{MSS}) gives also an analytic Frobenius type structure 
$$\mathcal{F}=(\Delta , E^{an}, R_{0}^{an}, R_{\infty}, \Phi^{an} , \bigtriangledown^{an} , g^{an})$$
on the simply connected domain $\Delta$. 
Universal deformations of this Frobenius type structure are defined in \cite[Definition 2.3.1]{Do1} and \cite{HeMa}.
The following results are shown and discussed in detail in \cite{Do1} in a slightly different situation, but the arguments in {\em loc. cit.} can be repeated almost verbatim here so we give only a sketch of the proofs.

We keep in this section the notations of section \ref{Bmodel}.
Let $\omega_{0}^{an}$ be the class of $\omega_{0}$ in $E^{an}$:  $\omega_{0}^{an}$ is $\bigtriangledown^{an}$-flat because $R(\omega_{0})=0$.

\begin{lemma} 
(1) The Frobenius type structure $\mathcal{F}$ has a universal deformation 
$$\widetilde{\mathcal{F}}=(N, \widetilde{E}^{an}, \widetilde{R}_{0}^{an}, \widetilde{R}_{\infty}, \widetilde{\Phi}^{an} , \widetilde{\bigtriangledown}^{an} , \widetilde{g}^{an})$$
parametrized by $N:=\Delta\times (\cit^{\mu -1},0)$.\\
(2) Let $\widetilde{\omega}_{0}^{an}$ be
the $\widetilde{\bigtriangledown}^{an}$-flat extension of  $\omega_{0}^{an}$. The period map 
$$\varphi_{\widetilde{\omega}_{0}^{an}}:\Theta_{N}\rightarrow \widetilde{E}^{an}$$
defined by 
$\varphi_{\widetilde{\omega}_{0}^{an}}(\xi )=-\widetilde{\Phi}^{an}_{\xi}(\widetilde{\omega}_{0}^{an})$
is an isomorphism which makes $N$ a Frobenius manifold.
\end{lemma}
\begin{proof} (1) We can use the adaptation of \cite[Theorem 2.5]{HeMa} given in \cite[Section 6]{Do1} because 
$$\omega_{0}^{an}, R_{0}^{an}(\omega_{0}^{an}),\cdots , (R_{0}^{an})^{\mu -1}(\omega_{0}^{an})$$
 generate $E^{an}$ and because $u_{0}:=1/u_{1}^{w_{1}}\cdots u_{n}^{w_{n}}$ is not equal to zero in $E^{an}$. (2) follows from (1) (see e.g. \cite[Theorem 4.5]{HeMa}). 
\end{proof}

\noindent The previous construction can be also done in the same way ''point by point''
(see \cite{DoSa2} and \cite{HeMa} and the references therein)
and this is the classical point of view: if $x\in \Delta$ one can attach to the Laurent polynomial  $F_{x}:=F(.,x)$ a Frobenius type structure on a point  $\mathcal{F}_{x}^{pt}$, a universal deformation $\widetilde{\mathcal{F}}_{x}^{pt}$ of it, again because $u_{0}$ and its powers generate $\cit [u,u^{-1}](\partial_{u_{i}} F_{x})$, and finally a Frobenius structure on $M:=(\cit^{\mu},0)$ with the help of the section $\omega_{0}$. We will call it ''the Frobenius structure attached to $F_{x}$''.
Let $\mathcal{F}_{x}$ ({\em resp.} $\widetilde{\mathcal{F}}_{x}$) be the germ of $\mathcal{F}$ ({\em resp.} $\widetilde{\mathcal{F}}$) at $x\in \Delta$ ({\em resp.} $(x,0)$).

\begin{proposition}
(1) The deformations  $\widetilde{\mathcal{F}}_{x}$ and $\widetilde{\mathcal{F}}_{x}^{pt}$ are isomorphic.\\                                   
(2) The period map defined by the flat extension of $\omega_{0}^{an}$ to $\widetilde{\mathcal{F}}_{x}$ is an isomorphism. This yields a
Frobenius structure on $M$ which is isomorphic to the one attached to $F_{x}$.
\end{proposition}
\begin{proof}
Notice first that $\widetilde{\mathcal{F}}_{x}^{pt}$ is a deformation of $\mathcal{F}_{x}$: this follows from the fact that 
$u_{0}$ does not belong to the Jacobian ideal of $f$: see \cite[section 7]{Do1}. Better, $\widetilde{\mathcal{F}}_{x}^{pt}$ is a universal deformation of 
 $\mathcal{F}_{x}$ because $\mathcal{F}_{x}$ is a deformation of $\mathcal{F}_{x}^{pt}$. This gives (1) because, by definition, two universal deformations of a same Frobenius type structure are isomorphic. (2) is then clear.  
\end{proof}

\noindent As a consequence, the universal deformations $\widetilde{\mathcal{F}}_{x}^{pt}$, $x\in\Delta$, are the germs of a same section, namely $\widetilde{\mathcal{F}}$. Thus, the Frobenius structure attached to $F_{x_{1}}$, $x_{1}\in\Delta$, can be seen as an analytic continuation of the one attached to $F_{x_{0}}$, $x_{0}\in\Delta$.

\subsection{Limit Frobenius manifolds}

\label{varfrob}

In order to construct limit Frobenius manifolds
we start from the limit structures given in section \ref{limitSTF}. We mimic the process explained in section \ref{constructionfrobman}: the main point is to find an unfolding of our limit Frobenius type structure $\overline{\fit}_{w}$ such that the associated period map is an isomorphism. 
To do this, we first unfold the quantum differential system $\overline{\mathcal{S}}_{w}$ (which is after all a vector bundle with connection) and then we use proposition 
\ref{MStoSTF}.

It should be emphasized that
the cases $\mu =n+1$ (manifold) and $\mu\geq n+2$ (orbifold) will yield different conclusions.

\subsubsection{Unfoldings of the limit structures}
The first step is thus to unfold the limit quantum differential system 
$$\overline{\mathcal{S}}_{w}=(\overline{H}, \overline{\nabla} , \overline{S}, n)$$ 
(see definition \ref{canonicallimitST}). A basis of global sections of $\overline{H}$ is $e=(e_{0},\cdots ,e_{\mu -1})$ where we put $e_{i}:=[\omega_{i}]$
(remember that $[\omega_{i}]$ denotes the class of $\omega_{i}$ in $\overline{H}$).
Recall the matrices $\overline{A}_{0}$ and $A_{\infty}$ defined in section \ref{limSTF}.

Define, for $i=0,\cdots ,\mu -1$, the matrices $C_{i}$ by

$$C_{i}(e_{j})=\left\{ \begin{array}{ll}
-\frac{1}{w^{w}}e_{\overline{i+j}} &  \mbox{if $i+j\geq \mu$ and $1+c_{\overline{i+j}}=c_{i}+c_{j}$,}\\
-e_{i+j} & \mbox{if $i+j\leq \mu -1$ and $c_{i+j}=c_{i}+c_{j}$,}\\
0 & \mbox{otherwise}
\end{array}
\right .$$
and put
$$\widetilde{A}_{0}(\underline{x})=(\alpha_{0}-1)x_{0}C_{0}-\mu C_{1}+(\alpha_{2}-1)x_{2}C_{2}+\cdots +(\alpha_{\mu -1}-1)x_{\mu -1}C_{\mu -1}$$
where $\underline{x}=(x_{0},\cdots ,x_{\mu -1})$ is a system of coordinates on $M=(\cit^{\mu},0)$ (with the previous notations, we have $x_{1}=x$). Notice that 
$-\mu C_{1}=\overline{A}_{0}$.\\

Let $\widetilde{H}$ be the trivial bundle on $\ppit^{1}\times M$ with basis $\widetilde{e}=(\widetilde{e}_{0},\cdots ,\widetilde{e}_{\mu -1})=(1\otimes e_{0},\cdots ,1\otimes
e_{\mu -1})$. Define on $\widetilde{H}$ the connection $\widetilde{\nabla}$ in the basis $\widetilde{e}$ is 
$$\left(\frac{\widetilde{A}_{0}(x)}{\theta}+A_{\infty}\right)\frac{d\theta}{\theta}+\theta^{-1}{\sum_{i=0}^{\mu
    -1}C_{i}dx_{i}}.$$ Define $\widetilde{S}$ on $\widetilde{H}$ by
$\widetilde{S}(\widetilde{e}_{i}, \widetilde{e}_{j})=\overline{S}(e_{i},e_{j})$,
this equality being extended by linearity.

\begin{proposition} (1) The tuple
 $$\widetilde{\mathcal{S}}_{w}=(M,\widetilde{H}, \widetilde{\nabla}, \widetilde{S},n)$$
is a quantum differential system  which unfolds $\overline{\mathcal{S}}_{w}$.\\
(2) Assume moreover that $w_{0}=w_{1}=\cdots =w_{n}=1$. Then the unfolding $\widetilde{\mathcal{S}}_{w}$ is universal.
\end{proposition}

\begin{proof} (1) We have to show that $\widetilde{\nabla}$ is flat and that $\widetilde{S}$ is $\widetilde{\nabla}$-flat.
The flatness is equivalent to the equalities
$$\frac{\partial C_{i}}{\partial x_{j}}=\frac{\partial C_{j}}{\partial x_{i}},\ [C_{i}, C_{j}]=0$$
$$[\widetilde{A}_{0}(x) , C_{i}]=0,\ \frac{\partial \widetilde{A}_{0}}{\partial x_{i}}+C_{i}=[A_{\infty}, C_{i}]$$
for all  $i,j$. Notice first that we have $C_{i}(e_{0})=-e_{i}$ for $i=0,\cdots ,\mu -1$. We have
$$C_{i}C_{j}(e_{k})=\left\{ \begin{array}{ll}
e_{i+j+k} &  \mbox{if $c_{i+j+k}=c_{i}+c_{j}+c_{k}$,} \\
e_{i+\overline{j+k}} & \mbox{if $1 +c_{i+\overline{j+k}}=c_{i}+c_{j}+c_{k}$,}\\
e_{\overline{i+j+k}} & \mbox{if $1 +c_{\overline{i+j+k}}=c_{i}+c_{j}+c_{k}$,}\\
e_{\overline{i+\overline{j+k}}} & \mbox{if $2 +c_{\overline{i+\overline{j+k}}}=c_{i}+c_{j}+c_{k}$}
\end{array}
\right .$$
This is symmetric in $i,j$ and thus $[C_{i}, C_{j}]=0$.
 Now if we define 
$$\widetilde{A}_{0}(x)=\sum_{i=0}^{\mu -1}([A_{\infty},C_{i}]-C_{i})x_{i}-\mu C_{1}$$
the conditions $\frac{\partial \widetilde{A}_{0}}{\partial x_{i}}+C_{i}=[A_{\infty}, C_{i}]$
for all  $i,j=0,\cdots ,\mu -1$ are obviously satisfied. But we have also $[A_{\infty}, C_{i}]=\alpha_{i}C_{i}$, 
because the condition 
$1+c_{\overline{i+j}}=c_{i}+c_{j}$ ({\em resp.} $c_{i+j}=c_{i}+c_{j}$) is equivalent to $\alpha_{\overline{i+j}}=\alpha_{i}+\alpha_{j}$
({\em resp.} $\alpha_{i+j}=\alpha_{i}+\alpha_{j}$),
hence $[\widetilde{A}_{0}(x),C_{i}]=0$
and the connection is flat. For the $\widetilde{\nabla}$-flatness of $\widetilde{S}$, it is enough to notice that $C_{i}^{*}=C_{i}$, $^{*}$ denoting the adjoint with respect to $\overline{S}_{w}$. This is shown using the kind of computations above. For the second assertion, notice that $\widetilde{A}_{0}(0)=\overline{A}_{0}$.\\
(2) If $w_{0}=\cdots =w_{n}=1$, $e_{0}$ induces a cyclic vector of $\overline{A}_{0}$. Hence, we can use \cite[p. 123]{HeMa}: the universality then follows from the fact that  
$(C_{i})_{i+1,1}=-1$ for all $i=0,\cdots ,\mu -1$.
\end{proof}

The quantum differential system $\widetilde{\mathcal{S}}_{w}$, with the help of proposition \ref{MStoSTF}, gives a Frobenius type structure on $M$,
$$\widetilde{\fit}_{w}=(M,\widetilde{E}, \widetilde{\bigtriangledown} , \widetilde{R}_{0}, \widetilde{R}_{\infty}, \widetilde{\Phi} , \widetilde{g})$$
the matrices of $\widetilde{R}_{0}$ and $\widetilde{R}_{\infty}$ being, in the obvious bases,
$\widetilde{A}_{0}$ and
$-A_{\infty}$.
  By definition, it is an unfolding of
$\overline{\fit}_{w}$.

\subsubsection{Construction of limit Frobenius manifolds}

In order to get a Frobenius manifold from the Frobenius type structure $\widetilde{\fit}_{w}$, we still need an invertible period map: its existence follows from the choice of the first columns of the matrices $C_{i}$.

\begin{corollary} \label{existFrobcan}  
(1) The period map 
$$\varphi_{\widetilde{e}_{0}} : TM\rightarrow \widetilde{E},$$
defined by $\varphi_{\widetilde{e}_{0}}(\xi )=-\widetilde{\Phi}_{\xi}(\widetilde{e}_{0})$, is an isomorphism and $\widetilde{e}_{0}$ is an eigenvector of $\widetilde{R}_{\infty}$.\\
(2) The section $\widetilde{e}_{0}$ defines, through the period map $\varphi_{\widetilde{e}_{0}}$ a Frobenius structure on $M$ which makes $M$ the  limit Frobenius manifold for which:\\
(a) the coordinates $(x_{0},\cdots ,x_{\mu -1})$ are $\bigtriangledown$-flat: one has $\bigtriangledown \partial_{x_{i}}=0$ for all  $i=0,\cdots ,\mu -1$,\\
(b) the product is constant in flat coordinates,\\
(c) the potential $\Psi$ is a polynomial of degree less or equal to $3$,\\
(d) the Euler vector field is $E=-(\alpha_{0}-1)x_{0}\partial_{x_{0}}+\mu\partial_{x_{1}}-(\alpha_{2}-1)x_{2}\partial_{x_{2}}-\cdots -(\alpha_{\mu -1}-1)x_{\mu -1}\partial_{x_{\mu -1}}$.
\end{corollary}
\begin{proof}
(1) Indeed, the period map $\varphi_{\widetilde{e}_{0}}$ is defined by  $\varphi_{\widetilde{e}_{0}}(\partial_{x_{i}})=-C_{i}(\widetilde{e}_{0})=\widetilde{e}_{i-1}$. Last, $\widetilde{e}_{0}$ is an eigenvector of $\widetilde{R}_{\infty}$ because $e_{0}$ is an eigenvector of $R_{\infty}$. Let us show (2): the isomorphism $\varphi_{\widetilde{e}_{0}}$ brings on $TM$ the structures on $\widetilde{E}$: (a) follows from the fact that the first column of the  matrices $C_{i}$ are constant and (b) from the fact that the matrices $C_{i}$ are constant because, by the definition of the product, $\varphi_{\widetilde{e}_{0}}(\partial_{x_{i}}*\partial_{x_{j}})=C_{i}(C_{j}(\widetilde{e}_{0}))$; (c) follows from (b) because, in flat coordinates,
$$g(\partial_{x_{i}}*\partial_{x_{j}},\partial_{x_{k}})=\frac{\partial^{3}\Psi}{\partial x_{i}\partial x_{j}\partial x_{k}}$$
where $g$ is the metric on $TM$ induced by $\widetilde{g}$.
Last, (d) follows from the definition of  $\widetilde{A}_{0}(x)$.
\end{proof}
 
\begin{remark} 
 If $w_{1}=\cdots =w_{n}=1$, the product is given by $\partial_{x_{i}}*\partial_{x_{j}}=\partial_{x_{i+j}}$ if $i+j\leq \mu -1$, $0$ otherwise, and we have
$$\Psi =\sum_{i,j, \; i+j\leq \mu -1}\frac{1}{6}x_{i}x_{j}x_{\mu -1-i-j}$$
up to a polynomial of degree less or equal to $2$. $\blacklozenge$
\end{remark}

\begin{remark}
 Of course, the period map can be an isomorphism for other
choices of the first columns of the matrices $C_{i}$: 
\begin{itemize}
\item the resulting Frobenius manifolds will be isomorphic to the
one given by corollary \ref{existFrobcan} if $w_{1}=\cdots =w_{n}=1$ (manifold case) because the Frobenius type structure $\widetilde{\fit}_{w}$ is a {\em universal} deformation of our limit Frobenius type structure $\overline{\fit}_{w}$ (see \cite{HeMa} and \cite[Theorem 3.2.1]{Do1}). 
This Frobenius structure is the one on
  $M:=H^{*}(\ppit^{n},\cit)$ given by the cup product and the Poincar\'e
  duality on each tangent spaces.
\item If there exists an $w_{i}$ such that $w_{i}\geq 2$ (orbifold case), 
one theoretically could get, starting from $\overline{\fit}_{w}$, several Frobenius manifolds (we have shown that there exists at least one), which can be difficult to compare because we loose the universality property here. However, the Frobenius manifold obtained in the previous corollary is the one on $M:=H^{*}_{\orb}(\ppit(w),\cit)$ given by the orbifold cup product and the Poincar\'e duality on each tangent spaces.
\end{itemize}$\blacklozenge$
\end{remark}

\subsection{Logarithmic Frobenius manifolds}\label{subsec:log}

A manifold $M$ is a {\em Frobenius manifold with logarithmic poles along the divisor} $D=\{x=0\}$ (for short a {\em logarithmic Frobenius manifold}) if $Der_{M}(\log D)$ is equipped with a metric, a multiplication and two global
logarithmic vector fields (the unit $e$ for the multiplication and the Euler vector field $E$), all these objects satisfying the usual compatibility relations, see \cite[Definition 1.4]{R}. We can also define a Frobenius manifold with logarithmic poles {\em without metric}: in this case, we still need a flat, torsionless connection, a symmetric Higgs field (that is a product) and two global logarithmic vector fields as before.  

There are two ways to construct such manifolds: the first one is to start from initial data, namely a logarithmic Frobenius type structure in the sense of \cite[Definition 1.6]{R}, and to unfold it, just as in section \ref{constructionfrobman}. 
This logarithmic Frobenius type structure will be obtained from a logarithmic quantum differential system, as in proposition \ref{MStoSTF}.
The second is to work directly with the big Gromov-Witten potential, as it is done in {\em loc. cit.} in the case of $\ppit^{n}$. We explore these two ways.

\subsubsection{Construction via unfoldings}\label{LogFrTySt}

Let $N=\cit$. We will denote the coordinate on $N$ by $x$ and we will put $D:=\{x=0\}$. The following definitions are borrowed from \cite{R}.

\begin{definition}\label{logSaitoStructure}
 A quantum differential system of weight $n$ on $\ppit^{1}\times N$ with {\em logarithmic poles along $D$} (for short a 
{\em logarithmic quantum differential system}) is a tuple $$(N, D, H^{log}, \nabla^{log} ,S^{log}, n)$$
where $H^{log}$ is a trivial bundle on $\ppit^{1}\times N$, $\nabla^{log}$ is a flat meromorphic connection on $H^{log}$ such that
$$\nabla^{log} (\Gamma (\ppit^{1}\times N, H^{log}))\subset \theta^{-1}\Omega_{\cit\times N}^{1}(\log((\{0\}\times \cit)\cup (\cit\times \{0\})))\otimes 
\Gamma (\ppit^{1}\times N, H^{log})$$
and $S^{log}$ is a $\nabla^{log}$-flat bilinear form as in definition \ref{defi:Saito,structure}.  
\end{definition}

In order to construct logarithmic Frobenius manifolds, we will need the following

\begin{definition}\label{logFTS}
A Frobenius type structure with {\em logarithmic pole along} $D$ (for short, a {\em logarithmic Frobenius type structure}) is a tuple
$$(N, D, E^{log}, \bigtriangledown^{log} , R_{0}^{log}, R_{\infty}^{log},\Phi^{log} ,g^{log})$$
\noindent where $E^{log}$ is a bundle on $N$, $R_{0}^{log}$ and $R_{\infty}^{log}$ are $\mathcal{O}_{N}$-linear endomorphisms of $E^{log}$,
 $$\Phi^{log} :E^{log}\rightarrow \Omega^{1}(\log (D))\otimes E^{log}$$ 
is a $\mathcal{O}_{N}$-linear map, $g^{log}$ is a metric on $E^{log}$, {\em i.e} a $\mathcal{O}_{N}$-bilinear form, symmetric and nondegenerate,
and $\bigtriangledown^{log}$ is a connection on $E^{log}$ with logarithmic pole along $D$, these object satisfying the compatibility relations of section \ref{FrobeniusSaito}. 
\end{definition}

\begin{remark} (1) One can also define in an obvious way a logarithmic quantum differential system and logarithmic Frobenius type structure {\em without metric}.\\
(2) As in section \ref{FrobeniusSaito}, a logarithmic quantum differential system determines a logarithmic Frobenius type structure (see \cite[proposition 1.10]{R})\\
(3) As before, we will work preferably in the algebraic category: $E^{log}$ will be a free $\cit [x]$-module {\em etc...}$\blacklozenge$
\end{remark}

 Proposition \ref{prop:nabla,basis,nonflat} and theorem \ref{basevarphi} suggests that we are not so far from a logarithmic quantum differential system. Indeed, with the notations of  section \ref{Bmodel} and forgetting the index $B$, $H^{log}$ will be obtained from an extension of $G_{0}$ as a free $\cit [x,\theta ]$-module (recall that $G_{0}$ is only a $\cit [x,x^{-1},\theta ]$-module). We can use for instance the $\cit [x,\theta ]$-submodule of $G_{0}$ generated by $\omega_{0},\cdots ,\omega_{\mu -1}$, and we thank C. Sevenheck for this suggestion: we will denote it by ${\cal L}_{0}$. Let ${\cal L}_{\infty}$ be the $\cit [x,\tau ]$-module
generated by $\omega_{0},\cdots ,\omega_{\mu -1}$ where, as usual, $\tau :=\theta^{-1}$. 
These two free modules give a trivial bundle $H^{log}$ equipped with a connection with the desired poles, thanks to  
theorem \ref{basevarphi}.
In order to define the metric $S^{log}$, extend the bilinear form $S$ defined in section \ref{dualite} to ${\cal L}_{0}$.
We will denote the resulting tuple by $\mathcal{S}^{log}_{w}$.

The logarithmic Frobenius type structure is then obtained as follows: put $E^{log}={\cal L}_{0}/\theta {\cal L}_{0}$. Define, 
as in section \ref{FrobeniusSaito}, the endomorphisms $R_{0}^{log}$ and $\Phi_{\xi}^{log}$ for any logarithmic vector field 
$\xi\in Der_{\cit}(\log D)$ and, using now the restriction of ${\cal L}_{\infty}$ at $\tau =0$, the endomorphisms $R_{\infty}^{log}$ and $\bigtriangledown_{\xi}^{log}$.
We get the flat bilinear symmetric form $g^{log}$ on $E^{log}$ putting
$$g^{log}([\omega_{i}], [\omega_{j}]):=\theta^{-n}S^{log}(\omega_{i}, \omega_{j})$$
where $[\, ]$ denotes the class in $E^{log}$. We will denote the resulting tuple by $\fit^{log}_{w}$.

\begin{proposition}\label{LogSTF} (1) The tuple $\mathcal{S}^{log}_{w}$ is a logarithmic quantum differential system if $w_{0}=\cdots =w_{n}=1$ and a 
logarithmic quantum differential system {\em without metric} otherwise.\\
(2) The tuple $\fit^{log}_{w}$ 
is a logarithmic Frobenius type structure if $w_{0}=\cdots =w_{n}=1$ and a logarithmic Frobenius type structure {\em without metric} otherwise.
\end{proposition}
\begin{proof}
By section \ref{dualite}, $S^{log}$ is not nondegenerate, unless $w_{0}=\cdots =w_{n}=1$. This gives (1) and (2) follows.
\end{proof}

\begin{corollary}\label{coro:logFrob,without,metric}
The section $\omega_{0}$ together with the tuple $\fit^{log}_{w}$ define a logarithmic Frobenius manifold if $w_{0}=\cdots =w_{n}=1$ and a logarithmic Frobenius manifold without metric otherwise. 
\end{corollary}
\begin{proof}
Define 
$$\varphi_{\omega_{0}}:Der_{\cit}(\log D)\rightarrow E^{log},$$
by $\varphi_{\omega_{0}}(\xi):=-\Phi_{\xi}^{log}(\omega_{0})$.
By theorem \ref{basevarphi}, the matrix of $\Phi_{x\partial_{x}}^{log}$ is $-{A_{0}(x)}{\mu^{-1}}$. Hence
 $\varphi_{\omega_{0}}|_{0}$ is injective and $\omega_{0} |_{0}$ and its images under iteration of the maps $\Phi_{x\partial_{x}}^{log}|_{0}$ generate $E^{log}|_{0}$. The result now follows from \cite[theorem 1.12]{R} because the section $\omega_{0}$ satisfies conditions (IC), (EC) and (GC) of {\em loc. cit.} and its restriction to $N-D$ is $\bigtriangledown^{log}$-flat (because $R_{\infty}^{log}(\omega_{0})=0$).
\end{proof}

\noindent If $w_{0}=\cdots =w_{n}=1$, we thus get a counterpart of the results obtained for $\ppit^{n}$, by a different method (see section below) in \cite[section 2]{R}. If there exists a weight $w_{i}$ such that $w_{i}\geq 2$, the construction of a logarithmic Frobenius manifold with metric using this method is still an open problem.

\begin{remark} One could of course consider different extensions of $G_{0}$ as a free $\cit [x,\theta ]$-module and start with a different 
logarithmic quantum differential system: for instance, it is possible to work with the lattice ${\cal L}_{0}^{\psi}$ such that the eigenvalues of the residue matrix of $\nabla_{\partial_{x}}$ at $x=0$ are contained in $]-1,0]$. It is easily checked that (with obvious notations) the section $\omega_{0}^{\psi}$ in ${\cal L}_{0}^{\psi}$ is flat but does not satisfy (GC) if $\mu\geq n+2$. The only section which satisfies (IC), (EC) and (GC) is $\omega_{n+1}^{\psi}$ but this one is not flat.  $\blacklozenge$
\end{remark}

\subsubsection{Construction via the Gromov-Witten potential}

In \cite{R}, Reichelt associates a logarithmic Frobenius manifold to a smooth projective variety, using
the Gromov-Witten potential. In this
section, we explain why his construction does not apply in the
orbifold case.

In order to simplify the notations, we focuse on
weighted projective spaces.  Put $M_{A}:=H^{*}_{\orb}(\ppit(w),\cit)$ and
let $(M_{A},H^{A},\nabla^{A},S^{A},n)$ be its big $A$-model quantum differential system 
(see Definition \ref{defi:big,A,D,module}).  We define the action
of $\Pic({\ppit(w)})$ on the trivial bundle $H^{A}\to \ppit^{1}\times
M_{A}$ as follows:
\begin{enumerate}
\item on the fibers of $H^{A}$ we define, for any $f\in F$ and $\alpha_{f}\in
  H^{*}(\ppit(w)_{S_{f}},\cit)$,
  \begin{align*}
\mathcal{O}(d)\cdot \alpha_{f}:=e^{2\pi\sqrt{-1} d f}\alpha_{f}
  \end{align*}
\item on $M_{A}= H^{*}_{\orb}(\ppit(w),\cit)$ we define 
  \begin{displaymath}
    \mathcal{O}(d)\cdot \left(\alpha\oplus\bigoplus_{f\in
        F/ \{0\}} \alpha_{f}\right):=(\alpha-2\pi\sqrt{-1}
    d.c_{1}(\mathcal{O}(1)))\oplus\bigoplus_{f\in F/ \{0\}}e^{2\pi\sqrt{-1} d.f}\alpha_{f}
  \end{displaymath}
\end{enumerate}
 
As in proposition \ref{prop:action,classes}, the quantum differential system is
equivariant with respect to this action so that we have a quotient
quantum differential system 
$(\mathcal{M}_{A},\widetilde{H}^{A},\widetilde{\nabla}^{A},\widetilde{S}^{A},n)$
where $\mathcal{M}_{A}:=M_{A}/\Pic(\ppit(w)).$ As the basis
$(\mathbf{1}_{f}P^{k})$ is not invariant for $f\neq 0$ with respect to
this action on $M_{A}$ (see Proposition \ref{prop:action,bases,P,1f}),
the associated coordinates $(t_{0},q=e^{t_{1}},t_{2}, \ldots
,t_{\mu-1})$ on $M_{A}$ are not coordinates on the quotient
$\mathcal{M}_{A}$.  Nevertheless, we can complete $(t_{0},q=e^{t_{1}},t_{2}, \ldots
,t_{n})$ in order to get a system of coordinates, denoted by
$\underline{\tau}=(t_{0},q=e^{t_{1}},t_{2}, \ldots ,t_{n},\tau_{n+1},
\ldots ,\tau_{\mu-1})$, on $\mathcal{M}_{A}$.

Put $\widetilde{E}^{A}:=\widetilde{H}^{A}\mid_{\{0\}\times
  \mathcal{M}_{A}}$.  If we want to repeat the argument given by
Reichelt in \S 2.1.1 \cite{R}, we should define the
metric using a ``infinitesimal period map" $T\mathcal{M}_{A}\to
\widetilde{E}^{A}$ which sends the vector field $\partial_{\tau_{i}}$
to $\mathbf{1}_{c_{i}}P^{r(i)}$ (cf \eqref{eq:17} for the notation).
This is not allowed in the orbifold case because for $c_{i}\neq 0$ the
cohomology class $\mathbf{1}_{c_{i}}P^{r(i)}$ does not define a global
section of the quotient bundle $\widetilde{H}^{A}\to \ppit^{1}\times
\mathcal{M}_{A}$.  

Natural global sections of $\widetilde{E}^{A}$ are
$({P^{\bullet_{\underline{\tau}}}}^{i})_{i\in\{0, \ldots ,\mu-1\}}$. But
proposition \ref{prop:pairingglobal} implies that the metric
degenerates at $q=0$. Hence as in corollary \ref{coro:logFrob,without,metric}, using
these global sections, we get a logarithmic Frobenius manifold without
metric on $\mathcal{M}_{A}$ in the orbifold case.

\

\noindent {\bf Acknowledgements.}  
We thank C. Sabbah for numerous discussions and for his support during
all these years. We would like also to thank H. Iritani, C.
Sevenheck, T.  Reichelt and I. de Gregorio for fruitful discussions
about this subject.  The second author is grateful to T. Coates for his explanations about 
Givental's framework in Luminy. He is also grateful to the scientific
board of the university Montpellier 2 and to the university of
Saint-Etienne for its hospitality.

\end{document}